\documentclass[a4paper,reqno,11pt]{amsart}
\usepackage[tmargin=40mm,bmargin=35mm,hmargin=25mm]{geometry}
\usepackage{amssymb}
\usepackage{amstext}
\usepackage{amsmath}
\usepackage{amscd}
\usepackage{amsthm}
\usepackage{amsfonts}
\usepackage{enumerate}
\usepackage{graphicx}
\usepackage{latexsym}
\usepackage{mathrsfs}
\input xy
\xyoption{all}
\usepackage{pstricks}
\usepackage{lscape}
\newtheorem{theorem}{Theorem}[section]

\newtheorem{corollary}[theorem]{Corollary}
\newtheorem{lemma}[theorem]{Lemma}
\newtheorem{definition-theorem}[theorem]{Definition-Theorem}
\newtheorem{proposition}[theorem]{Proposition}

\theoremstyle{definition}
\newtheorem{definition}[theorem]{Definition}

\newtheorem{example}[theorem]{Example}

\newtheorem{question}[theorem]{Question}

\numberwithin{equation}{section}
\newcommand{\add}{\mathsf{add}\hspace{.01in}}

\newcommand{\Kb}{\mathsf{K}^{\rm b}}

\renewcommand{\mod}{\mathsf{mod}\hspace{.01in}}
\newcommand{\proj}{\mathsf{proj}\hspace{.01in}}

\newcommand{\thick}{\mathsf{thick}\hspace{.01in}}

\newcommand{\Aut}{\operatorname{Aut}\nolimits}

\newcommand{\End}{\operatorname{End}\nolimits}
\newcommand{\Hom}{\operatorname{Hom}\nolimits}
\newcommand{\rad}{\operatorname{rad}\nolimits}
\newcommand{\soc}{\operatorname{soc}\nolimits}

\newcommand{\tilt}{\mbox{\rm tilt}\hspace{.01in}}
\newcommand{\silt}{\mbox{\rm silt}\hspace{.01in}}
\newcommand{\nsilt}{\mbox{\rm silt${}^\nu$}\hspace{.01in}}

\newcommand{\ssilt}[1]{{#1}\operatorname{-silt}}
\newcommand{\ttilt}[1]{{#1}\operatorname{-tilt}}

\newcommand{\xto}{\xrightarrow}
\newcommand{\kD}{\mathbb{D}}

\begin{document}
\title[Tilting-discrete triangulated algebras not silting-discrete]{Examples of tilting-discrete self-injective algebras which are not silting-discrete}
\author{Takahide Adachi}
\address{T.~Adachi: Faculty of Global and Science Studies, Yamaguchi University, 1677-1 Yoshida, Yamaguchi 753-8541, Japan}
\email{tadachi@yamaguchi-u.ac.jp}
\thanks{T.~Adachi is supported by JSPS KAKENHI Grant Number JP20K14291.}
\author{Ryoichi Kase}
\address{R.~Kase: Faculty of Informatics, Okayama University of Science, 1-1 Ridaicho, Kita-ku, Okayama 700-0005, Japan}
\email{r-kase@mis.ous.ac.jp}
\thanks{R.~Kase is is supported by JSPS KAKENHI Grant Number JP17K14169.}
\dedicatory{Dedicated to Professor Susumu Ariki on the occasion of his 60th birthday}

\begin{abstract}
In this paper, we introduce the notion of $\nu$-stable silting-discrete algebras, which unify silting-discrete algebras and tilting-discrete self-injective algebras, where $\nu$ is a triangle auto-equivalence of the bounded homotopy category of finitely generated projective modules. Moreover, we give an example of tilting-discrete self-injective algebras which are not silting-discrete.
\end{abstract}
\maketitle

\section{Introduction}
The study of derived categories is considered as an important subject in various mathematics, for example, ring theory, representation theory, algebraic geometry and mathematical physics.
In the representation theory of algebras, since the equivalences of derived categories preserve many homological properties, it is a natural problem to determine the derived equivalence class of a given algebra.
It is a well-known result (\cite{Ri89}) that derived equivalences are controlled by tilting objects.
Hence the problem above is reduced to finding all tilting objects for an algebra.

Recently, mutation theory has been intensely studied in the representation theory of algebras.
Mutation is an operation to construct a new object from an original one by exchanging direct summands.
As a typical example, for a symmetric algebra, mutations of tilting objects are also tilting, known as Okuyama--Rickard complexes.
Unfortunately, for any algebra, the class of tilting objects is not necessarily closed under mutations.
Aihara--Iyama (\cite{AI12}) shows that mutations of silting objects are always silting objects, and hence  mutations make infinitely many silting objects from a given silting object.
Silting objects are introduced by Keller--Vossieck (\cite{KV88}) as a generalization of tilting objects in order  to study bounded $t$-structures on derived categories.

We may expect silting connectedness, that is, any two silting objects are obtained from each other by iterated mutation.
However, Aihara--Grant--Iyama and recently Dugas (\cite{Du21}) give examples of algebras which do not satisfy silting connectedness.
In \cite{Ai13}, Aihara introduce the notion of silting-discrete algebras, which gives a reasonable class of finite dimensional algebras satisfying silting connectedness. 
A finite dimensional algebra is called a \emph{silting-discrete algebra} if for each positive integer $d$, the set of isomorphism classes of basic $d$-term silting objects of the bounded homotopy category of finitely generated projective modules is finite.
As nice properties of silting-discrete algebras, bounded $t$-structures correspond bijectively with silting objects (\cite{KY14,AMY19}) and hence the stability space (in the sense of Bridgeland) of the bounded derived category is contractible (\cite{PSZ18,AMY19}).

As mentioned above, for any algebra, mutations of tilting objects are not necessarily tilting.
However, for a self-injective algebra, Chan--Koenig--Liu (\cite{CKL15}) introduce the notion of $\nu$-stable mutation and show that $\nu$-stable mutations of tilting objects are also tilting, where $\nu$ is a Nakayama functor.
It is shown (\cite{AM17}) that tilting-discrete self-injective algebras, which are a tilting analog of silting-discrete algebras, satisfy a property that any two tilting objects are obtained from each other by iterated $\nu$-stable mutation.

In this paper, we discuss a unification of silting-discrete algebras and tilting-discrete self-injective algebras.
Moreover, we give an example of tilting-discrete self-injective algebras that are not silting-discrete.
Let $A$ be a finite dimensional algebra and $\mathcal{T}:=\Kb(\proj A)$ the bounded homotopy category of finitely generated projective $A$-modules.
For a triangle auto-equivalence $\nu$ on $\mathcal{T}$,
we introduce the notion of $\nu$-stable silting-discrete algebras, that is, algebras with finitely many $d$-term $\nu$-stable silting objects of $\mathcal{T}$ for each $d>0$.
Remark that $\nu$-stable silting objects of $\mathcal{T}$ is a generalization of tilting objects for self-injective algebras (see Proposition \ref{prop:selfinjective-tilting}).
The following theorem is one of our main results, which is an analog of \cite[Theorem 1.2]{AM17}.

\begin{theorem}\label{thmA}
Let $A$ be a finite dimensional algebra and $\mathcal{T}:=\Kb(\proj A)$. 
Assume that $\mathcal{T}$ admits a triangle auto-equivalence $\nu$.
Then the following statements are equivalent:
\begin{itemize}
\item[(1)] $A$ is $\nu$-stable silting-discrete.
\item[(2)] For each object $M$ obtained by finite sequence of minimal $\nu$-stable mutations from $A$, the set of isomorphism classes of basic $\nu$-stable silting objects $N$ of $\mathcal{T}$ satisfying $M\ge N\ge M[1]$ is a finite set, where $X\ge Y$ means $\Hom_{\mathcal{T}}(X,Y[i])=0$ for all $i>0$.
\end{itemize}
\end{theorem}

Remark that Theorem \ref{thmA} is extended to the case of triangulated categories (see Theorem \ref{thm:equiv-nu-silting-discrete}).

For a symmetric algebra, all silting objects are tilting objects.
Hence tilting-discrete symmetric algebras are silting-discrete.
This result is generalized to weakly symmetric algebras as follows.

\begin{theorem}[Theorem \ref{thm:weakly-symmetric-silting-discrete}]\label{thmB}
Let $A$ be a weakly symmetric algebra and $\mathcal{T}:=\Kb(\proj A)$.
Let $\nu$ be a Nakayama functor.
Then the following statements are equivalent:
\begin{itemize}
\item[(1)] $A$ is silting-discrete.
\item[(2)] $A$ is $\nu$-stable silting-discrete.
\item[(3)] $A$ is tilting-discrete.
\end{itemize}
In this case, all silting objects are tilting.
\end{theorem}

Independently of the present work, the same result is obtained by August--Dugas \cite{AD21}.

In \cite{AM17}, it is shown that preprojective algebras of Dynkin type are tilting-discrete self-injective algebras.
As an application of Theorem \ref{thmB}, we show that, if $A$ is the preprojective algebra of one of Dynkin diagrams $\mathbf{D}_{2n}(n\ge 2)$, $\mathbf{E}_{7}$ and $\mathbf{E}_{8}$, then it is silting-discrete.
However, we do not know whether each tilting-discrete self-injective algebra is silting-discrete.
Now we propose a natural question.
\begin{question}
Is a tilting-discrete self-injective algebra always silting-discrete?
\end{question}

One of our aims of this paper is to give two counterexamples for the question above.
The first counterexample is as follows.

\begin{theorem}[Theorem \ref{thm:local-tilting}]
Let $A$ be a basic connected non-semisimple self-injective algebra over an algebraically field and let $\nu$ be its Nakayama functor. Assume that $A$ is $\nu$-cyclic.
Then there exists a self-injective algebra $\widetilde{A}$ such that 
\begin{itemize}
\item it is not silting-discrete,
\item $\{ \widetilde{A}[i]\mid i\in\mathbb{Z} \}$ coincides with the set of isomorphism classes of all basic tilting objects for $\widetilde{A}$.
In particular, $\widetilde{A}$ is tilting-discrete.
\end{itemize}
\end{theorem}

The second counterexample is as follows.
Let $n,m$ be positive integers and let $K$ be an algebraically closed field.
We denote by $A_{n,m}$ the stable Auslander algebra of a self-injective Nakayama $K$-algebra with $m$ simple modules (up to isomorphism) and Loewy length $n$.
It is known that $A_{n,m}$ is always a self-injective algebra.

\begin{theorem}[Theorem \ref{thm:counterexample-question}]
Let $n,m\ge 5$ be integers with $\gcd(n-1,m)=1$.
Assume that $n$ is odd and $m$ is not divisible by the characteristic of $K$.
Then $A_{n,m}$ is a tilting-discrete algebra but not silting-discrete.
\end{theorem}

\subsection*{Notation}
Let $K$ be a field and $\kD:=\Hom_{K}(-,K)$.
Throughout this paper, $\mathcal{T}$ is a $K$-linear Hom-finite Krull-Schmidt triangulated category with shift functor $[1]$.
For an object $M$ of $\mathcal{T}$, we denote by $\add(M)$ the smallest full subcategory of $\mathcal{T}$ which contains $M$ and which is closed under taking finite direct sums and direct summands, and by $\thick M$ the smallest triangulated full subcategory of $\mathcal{T}$ which contains $M$ and which is closed under taking direct summands.
For full subcategories $\mathcal{X},\mathcal{Y}$ of $\mathcal{T}$, we define $\mathcal{X}\ast\mathcal{Y}$ as the full subcategory of $\mathcal{T}$ consisting of $T\in\mathcal{T}$ which admits a triangle $X\to T\to Y\to X[1]$ with $X\in\mathcal{X}$ and $Y\in\mathcal{Y}$.

\section{$\nu$-stable silting theory}
In this section, we introduce $\nu$-stable silting mutation theory, which unifies silting mutation theory (\cite{AI12}) and tilting mutation theory of self-injective algebras (\cite{CKL15}, \cite{AM17}), where $\nu$ is a triangle auto-equivalence.
Let $\mathcal{T}$ be a $K$-linear Hom-finite Krull-Schmidt triangulated category with shift functor $[1]$.
Assume that $\mathcal{T}$ has a triangle auto-equivalence $\nu:\mathcal{T} \to \mathcal{T}$.

\subsection{$\nu$-stable objects}
In this subsection, we recall the notion of $\nu$-stable objects, which plays an important role in this paper.

\begin{definition}
An object $M$ of $\mathcal{T}$ is said to be \emph{$\nu$-stable} if $\nu M \cong M$ holds.
\end{definition}

Let $M$ be a basic $\nu$-stable object of $\mathcal{T}$.
We decompose $M$ as $M=\oplus_{i\in I}M_{i}$, where $M_{i}$ is indecomposable.
Then for each $i\in I$, there uniquely exists $j\in I$ such that $\nu M_{i} \cong M_{j}$ because $\nu$ preserves indecomposability. 
Define a permutation $v_{M}:I\to I$ as $\nu M_{i}\cong M_{v_{M}(i)}$.
Now we introduce two classes of $\nu$-stable objects.

\begin{definition}
Let $M$ be a basic $\nu$-stable object of $\mathcal{T}$.
\begin{itemize}
\item[(1)] We call $M$ a \emph{weakly symmetric $\nu$-stable object} if $v_{M}$ is an identity map.
\item[(2)] We call $M$ a \emph{symmetric $\nu$-stable object} if the restriction $\nu|_{\add M}$ is functorial isomorphic to the identity functor.
\end{itemize}
\end{definition}

For simplicity, we omit the word ``$\nu$-stable'' in (weakly) symmetric $\nu$-stable objects.
Note that all symmetric objects are weakly symmetric.
Moreover, if $M$ is a symmetric object, then each object of $\thick M$ is $\nu$-stable.

Under the condition that $\nu$ is a Serre functor (i.e., there exists a bifunctorial isomorphism
\begin{align}\label{seq:serre}
\Hom_{\mathcal{T}}(X,Y)\cong\kD\Hom_{\mathcal{T}}(Y,\nu X)
\end{align}
for each $X,Y\in\mathcal{T}$), we obtain the following result.

\begin{proposition}\label{prop:object-algebra}
Assume that $\nu$ is a Serre functor.
If $M$ is a basic $\nu$-stable $($respectively, weakly symmetric, symmetric$)$ object, then $\End_{\mathcal{T}}(M)$ is a self-injective $($respectively, weakly symmetric, symmetric$)$ algebra. 
\end{proposition}
\begin{proof}
Let $M$ be a $\nu$-stable object of $\mathcal{T}$.
By \eqref{seq:serre}, we have $\End_{\mathcal{T}}(M)\cong \kD\End_{\mathcal{T}}(M)$ as a left $\End_{\mathcal{T}}(M)$-module.
Hence $\End_{\mathcal{T}}(M)$ is self-injective.
Next, we assume that $M$ is weakly symmetric. Let $M_{i}$ be an indecomposable direct summand of $M$.
Then we obtain $\Hom_{\mathcal{T}}(M_{i},M)\cong \kD\Hom_{\mathcal{T}}(M,\nu M_{i})\cong \kD\Hom_{\mathcal{T}}(M,M_{v_{M}(i)})\cong \kD\Hom_{\mathcal{T}}(M,M_{i})$ as a left $\End_{\mathcal{T}}(M)$-module.
Therefore $\End_{\mathcal{T}}(M)$ is weakly symmetric.
Finally, we assume that $M$ is symmetric.
Since $\nu|_{\add M}$ is functorial isomorphic to the identity functor, we have $\End_{\mathcal{T}}(M)\cong \kD\End_{\mathcal{T}}(M)$ as an $\End_{\mathcal{T}}(M)$-$\End_{\mathcal{T}}(M)$-bimodule.
Consequently, $\End_{\mathcal{T}}(M)$ is symmetric.
\end{proof}

\subsection{$\nu$-stable silting objects}
We start this subsection with recalling the definition of silting objects.

\begin{definition}
An object $M$ of $\mathcal{T}$ is called a \emph{silting} (respectively, \emph{tilting}) \emph{object} of $\mathcal{T}$ if $\mathcal{T}=\thick M$ and $\Hom_{\mathcal{T}}(M,M[i])=0$ for all $i>0$ (respectively, $i\neq 0$).
We denote by $\silt \mathcal{T}$ (respectively, $\tilt\mathcal{T}$, $\silt^{\nu}\mathcal{T}$) the set of isomorphism classes of basic silting (respectively, tilting, $\nu$-stable silting) objects of $\mathcal{T}$.
\end{definition}

Recall the partial order on $\silt\mathcal{T}$.
For objects $M,N$ of $\mathcal{T}$, we write $M \ge N$ if $\Hom_{\mathcal{T}}(M,N[i])=0$ for all $i>0$.
Then $(\silt\mathcal{T}, \ge)$ is a partially ordered set by \cite[Theorem 2.11]{AI12}.
Moreover, by the restriction, $\ge$ gives a partial order on $\silt^{\nu}\mathcal{T}$.
For each $M\in\silt\mathcal{T}$ and $d\in\mathbb{Z}_{\ge 0}$, let $\ssilt{(d+1)_{M}}\mathcal{T} :=\{ N\in \silt\mathcal{T} \mid M \ge N \ge M[d]\}$. 
Note that, for $M,N\in \silt\mathcal{T}$, $M\geq N \geq M[d]$ if and only if $N\in \add M\ast\add M[1]\ast \cdots \ast\add M[d]$ (for example, see \cite[Lemma 3.6]{AMY19}).

In \cite{ANR13} and also \cite[Theorem A.4]{Ai13}, it is shown that, for a finite dimensional self-injective algebra $A$ over an algebraically closed field, 
all $\nu$-stable silting objects of the bounded homotopy category $\Kb(\proj A)$ are tilting objects, where $\nu:=\kD\Hom_{A}(-,A)$ is a Serre functor.
Moreover the converse also holds.
We discuss an analog of their result.   

\begin{proposition}\label{prop:nu-silting-tilting}
Assume that $\nu$ is a Serre functor.
Then all $\nu$-stable silting objects of $\mathcal{T}$ are tilting.
\end{proposition}
\begin{proof}
Let $M$ be a $\nu$-stable silting object of $\mathcal{T}$.
It is enough to show that $\Hom_{\mathcal{T}}(M,M[i])=0$ for all $i<0$.
For each integer $i$, we have isomorphisms 
\begin{align}
\Hom_{\mathcal{T}}(M,M[i])\cong \kD\Hom_{\mathcal{T}}(M[i],\nu M)\cong \kD\Hom_{\mathcal{T}}(M[i],M) \cong \kD\Hom_{\mathcal{T}}(M,M[-i]). \notag
\end{align}
Since $M$ is silting, we obtain $\Hom_{\mathcal{T}}(M,M[i])=0$ for each negative integer $i$.
\end{proof}

Note that the converse in Proposition \ref{prop:nu-silting-tilting} does not necessarily hold.
Indeed, we give a characterization of algebras that all tilting objects are $\nu$-stable silting.
By the characterization, non-semisimple hereditary algebras have a tilting object which is not $\nu$-stable.
Recall that, by \cite[Corollary 3.9]{Che11}, a finite dimensional algebra $A$ is an Iwanaga--Gorenstein algebra if and only if the bounded homotopy category $\Kb(\proj A)$ has a Serre functor $\nu$.

\begin{proposition}\label{prop:selfinjective-tilting}
Let $A$ be a finite dimensional Iwanaga--Gorenstein algebra over an algebraically closed field and $\nu$ the Serre functor.
Then the following statements are equivalent.
\begin{itemize}
\item[(1)] $A$ is self-injective.
\item[(2)] All tilting objects of $\Kb(\proj A)$ are $\nu$-stable.
\item[(3)] $A$ is a $\nu$-stable object of $\Kb(\proj A)$.
\item[(4)] $\Kb(\proj A)$ has a $\nu$-stable silting object.
\end{itemize}
\end{proposition}
\begin{proof}
(1)$\Rightarrow$(2) follows from \cite[Theorem A.4]{Ai13}.
(2)$\Rightarrow$(3)$\Rightarrow$(4) is clear.
We show (4)$\Rightarrow$(1). Let $T$ be a $\nu$-stable silting object of $\Kb(\proj A)$. 
Then $B:=\End_{\Kb(\proj A)}(T)$ is a self-injective algebra by Proposition \ref{prop:object-algebra}.
Since $T$ is a tilting object of $\Kb(\proj A)$ by Proposition \ref{prop:nu-silting-tilting}, $B$ is derived equivalent to $A$. 
Hence the assertion follows from \cite[Theorem 2.1]{ANR13}.
\end{proof}
 
As an application of Proposition \ref{prop:nu-silting-tilting}, we have the following corollary.
\begin{corollary}
Assume that $\nu$ is a Serre functor.
Let $M$ be a symmetric object of $\mathcal{T}$.
Then all silting objects of $\thick M$ are tilting objects of $\thick M$.
\end{corollary}
\begin{proof}
Since $M$ is symmetric, all objects in $\thick M$ are $\nu$-stable.
Hence the assertion follows from Proposition \ref{prop:nu-silting-tilting}.
\end{proof}

\subsection{$\nu$-stable mutations}
Let us start this subsection by recalling the notion of $\nu$-stable mutations.
For basic results on mutations of silting objects, we refer to \cite{AI12}. 
In this subsection, we do not necessarily assume that $\nu$ is a Serre functor.

Recall the definition of minimal left approximations.
Let $f: X\to Z$ be a morphism.
We say that $f$ is \emph{left minimal} if each $h\in \End_{\mathcal{T}}(Z)$ with $hf=f$ is an isomorphism.
Let $N$ be an object of $\mathcal{T}$.
We call $f$ a \emph{left $\add N$-approximation} of $X$ if $Z\in\add N$ and $\Hom_{\mathcal{T}}(f,N)$ is surjective.
A left $\add N$-approximation $f$ is said to be \emph{minimal} if it is left minimal.
Dually, we define a right minimal morphism, a right $\add N$-approximation and a minimal right $\add N$-approximation. 
We collect some results for approximations. 
The following lemma is a basic result in mutation theory.

\begin{lemma}\label{lem:approximation-result}
Let $N$ be an object of $\mathcal{T}$ with $\Hom_{\mathcal{T}}(N,N[1])=0$.
Let 
\begin{align}
X\xto{f}Y\xto{g}Z\to X[1] \notag
\end{align}
be a non-split triangle.
Then the following statements are equivalent.
\begin{itemize}
\item[(1)] $X$ is indecomposable, $f$ is a minimal left $\add N$-approximation and $\Hom_{\mathcal{T}}(N,X[1])=0$. 
\item[(2)] $Z$ is indecomposable, $g$ is a minimal right $\add N$-approximation and $\Hom_{\mathcal{T}}(Z,N[1])=0$.
\end{itemize}
\end{lemma}

We have an easy observation for left minimal approximations.

\begin{lemma}\label{lem:isom-triang}
Let $N$ be an object of $\mathcal{T}$. Let $X\xto{f}Y\to Z\to X[1]$ and $X'\xto{f'}Y'\to Z'\to X'[1]$ be triangles with $f,f'$ minimal left $\add N$-approximations.
For an isomorphism $\varphi:X\to X'$, there exist isomorphisms $\varphi':Y\to Y'$ and $\varphi'':Z\to Z'$ such that the following diagram commutes:
\begin{align}
\xymatrix{
X\ar[r]^-{f}\ar[d]_-{\varphi}^-{\cong}&Y\ar[r]\ar[d]_-{\varphi'}^-{\cong}&Z\ar[r]\ar[d]_-{\varphi''}^-{\cong}&X[1]\ar[d]^-{\cong}\\
X'\ar[r]^-{f'}&Y'\ar[r]&Z'\ar[r]& X'[1].
}\notag
\end{align}
Moreover, if $N$ and $X$ are $\nu$-stable, then so are $Y$ and $Z$. 
\end{lemma}
\begin{proof}
The first assertion follows from basic properties of triangulated categories and minimal left approximations.
We show the second assertion. We assume that $X$ and $N$ are $\nu$-stable.
Since $\nu$ is a triangle auto-equivalence, the morphism $\nu f$ is also a minimal left $\add N$-approximation.
Hence the second assertion follows from the first assertion.
\end{proof}

Let $M$ be a basic object of $\mathcal{T}$ with $M=X\oplus N$.
Take a minimal left $\add N$-approximation $f: X\to Y$ and a triangle 
\begin{align}
X\xto{f}Y \to Z \to X[1].\notag
\end{align}
Then $\mu_{X}(M):=Z\oplus N$ is called a (left) mutation of $M$ with respect to $X$.
Moreover, the mutation $\mu_{X}(M)$ is said to be \emph{irreducible} if $X$ is indecomposable.
Mutations of silting objects have the following nice property.

\begin{proposition}\cite[Theorem 2.31 and Proposition 2.33]{AI12}\label{prop:silting-mutation}
Let $M=X\oplus N$ be a basic silting object.
Then $\mu_{X}(M)$ is also basic silting. 
Moreover, if $X\neq 0$, then $M>\mu_{X}(M)$.
\end{proposition}

In the following, we introduce the notion of $\nu$-stable mutations, which is an analog of mutations of tilting objects for self-injective algebras (see \cite[\S5]{CKL15}).
We call $\mu_{X}(M)$ a \emph{$\nu$-stable mutation} if $M$ and $X$ are $\nu$-stable. Note that if $M=X\oplus N$ is $\nu$-stable, then we obtain that $X$ is $\nu
$-stable if and only if $N$ is $\nu$-stable. By Lemma \ref{lem:isom-triang} and Proposition \ref{prop:silting-mutation}, we have the following result.

\begin{proposition}\label{prop:nu-silting-mutation}
Let $M=X\oplus N$ be a basic $\nu$-stable silting object with $X$ a $\nu$-stable object.
Then $\mu_{X}(M)$ is also a basic $\nu$-stable silting object.
\end{proposition}
\begin{proof}
Since $X$ and $N$ are $\nu$-stable, so is $\mu_{X}(M)$ by Lemma \ref{lem:isom-triang}.
Thus the assertion follows from Proposition \ref{prop:silting-mutation},
\end{proof}

Now we define irreducible $\nu$-stable mutations. 

\begin{definition}
Let $M$ be a $\nu$-stable object.
\begin{itemize}
\item[(1)] A non-zero $\nu$-stable direct summand $X$ of $M$ is said to be \emph{minimal} if there exists no non-zero proper $\nu$-stable direct summand $X'$ of $X$.
\item[(2)] Assume that $M$ is basic. If $X$ is a minimal $\nu$-stable direct summand of $M$, then we call $\mu_{X}(M)$ an \emph{irreducible $\nu$-stable mutation} of $M$ with respect to $X$.
\end{itemize}
\end{definition}

Let $M=X\oplus N$ be a weakly symmetric object. Since each indecomposable direct summand of $M$ is $\nu$-stable, we obtain that $X$ is minimal $\nu$-stable if and only if it is indecomposable. Thus irreducible mutations coincides with irreducible $\nu$-stable mutations.

\begin{proposition}\label{prop:wekaly-symmetric-mutation}
Each $\nu$-stable mutation of a weakly symmetric silting object is also weakly symmetric silting.
\end{proposition}
\begin{proof}
Let $M=X\oplus N$ be a weakly symmetric silting object and take a triangle $X\xto{f}Y\to Z\to X[1]$ with $f$ a minimal left $\add N$-approximation.
By Proposition \ref{prop:silting-mutation}, $\mu_{X}(M):=Z\oplus N$ is a silting object. Thus it is enough to show that $\mu_{X}(M)$ is weakly symmetric.
We decompose $X$ as $X=\oplus_{i\in I}X_{i}$, where $X_{i}$ is indecomposable.
For each $i\in I$, take a minimal left $\add N$-approximation $f_{i}:X_{i}\to Y_{i}$ and a triangle $X_{i}\xto{f_{i}}Y_{i}\to Z_{i}\to X_{i}[1]$.
By Lemmas \ref{lem:approximation-result} and \ref{lem:isom-triang}, $Z_{i}$ is indecomposable and $\nu$-stable respectively.
Hence $\oplus_{i\in I} Z_{i}$ is weakly symmetric.
On the other hand, since $\oplus_{i\in I}f_{i}$ is a minimal left $\add N$-approximation, it follows from Lemma \ref{lem:isom-triang} that $Z\cong\oplus_{i\in I} Z_{i}$ and hence $\mu_{X}(M)$ is weakly symmetric.
\end{proof}

In the rest of this subsection, we study combinatorial properties for $\nu$-stable mutations.
The following lemma plays an important role in this section. 

\begin{lemma}\cite[Propositions 2.24 and 2.36]{AI12}\label{lem:int-mutation}
Fix an integer $d\ge 1$. Let $M$ be a basic silting object and $N\in \add M \ast \add M[1]\ast\cdots \ast \add M[d]$. 
Then the following statements hold.
\begin{itemize}
\item[(1)] For each $l\in [1,d]$, there exists a triangle
\begin{align}
M_{l}'\xto{f}N\xto{g}M_{l}''\xto{h}M_{l}'[1] \notag
\end{align}
such that $f$ is a minimal right $(\add M\ast \cdots \ast \add M[l-1])$-approximation, $g$ is a minimal left $(\add M[l]\ast \cdots \ast \add M[d])$-approximation and $h$ is a radical of $\mathcal{T}$.
\item[(2)] If $N\in \ssilt{(d+1)_{M}}\mathcal{T}\setminus \ssilt{d_{M}}\mathcal{T}$, then $M_{d}''\neq 0$. 
Moreover, we have $M>\mu_{X}(M)\ge N$ for each basic non-zero direct summand $X$ of $M_{d}''[-d]$. 
\end{itemize}
\end{lemma}

The lemma above induces the following properties of $\nu$-stable mutations.

\begin{proposition}\label{prop:local-property-nu-mutation}
Let $M,N$ be basic silting objects with $M>N$.
Assume that one of the following two conditions is satisfied:
\begin{itemize}
\item[(a)] $M$ and $N$ are $\nu$-stable.
\item[(b)] $M$ is weakly symmetric.
\end{itemize}
Then the following statements hold.
\begin{itemize}
\item[(1)] There is a minimal $\nu$-stable direct summand $X$ of $M$ such that $M>\mu_{X}(M)\ge N$.  
\item[(2)] If the set 
\begin{align}
\silt^{\nu}[N,M]:=\{ L\in \silt^{\nu}\mathcal{T} \mid N\le L\le M\} \notag
\end{align}
is finite, then $N$ can be obtained from $M$ by iterated irreducible $\nu$-stable mutation. 
In particular, if $\textnormal{(b)}$ is satisfied, then all objects in $\silt^{\nu}[N,M]$ are weakly symmetric.
\end{itemize}
\end{proposition}
\begin{proof}
(1) Let $M>N$ be basic silting objects satisfying (a) or (b).
By \cite[Proposition 2.23]{AI12}, we have $N\in \ssilt{(d+1)_{M}}\mathcal{T}\setminus \ssilt{d_{M}}\mathcal{T}$ for some $d\ge 1$.
Then it follows from Lemma \ref{lem:int-mutation} that there exists a minimal left $\add M[d]$-approximation $g:N\to M_{d}''$ with $M_{d}''\neq 0$.
We show that $M_{d}''$ is $\nu$-stable.
If (a) is satisfied, then the assertion follows from Lemma \ref{lem:isom-triang}.
On the other hand, if (b) is satisfied, then the assertion follows from the fact that each indecomposable direct summand of $M$ is $\nu$-stable. 
Hence for both cases, $M_{d}''$ is $\nu$-stable.
Taking a minimal $\nu$-stable direct summand $X$ of $M''_{d}[-d]$, we have $M>\mu_{X}(M)\ge N$ by Lemma \ref{lem:int-mutation}(2).

(2) If $(a)$ (respectively, (b)) is satisfied, then $\mu_{X}(M)$ in (1) is also $\nu$-stable (respectively, weakly symmetric) by Proposition \ref{prop:nu-silting-mutation} (respectively, Proposition \ref{prop:wekaly-symmetric-mutation}).
By repeated use of (1), we have a sequence of irreducible $\nu$-stable mutations
\begin{align}
M>L_{1}>L_{2}>\cdots (\ge N) \notag 
\end{align}
in $\silt^{\nu}[N,M]$.
Since the set $\silt^{\nu}[N,M]$ is finite, there exists an integer $n>0$ such that $L_{n}=N$. 
Hence we have the assertion.
\end{proof}

As an analog of \cite[Theorem 2.35]{AI12}, we compare the Hasse quiver of $(\silt^{\nu} \mathcal{T},\ge)$ and the mutation quiver $Q(\silt^{\nu}\mathcal{T})=(Q_{0},Q_{1})$ defined as 
\begin{align}
&Q_{0}:=\nsilt\mathcal{T}, \notag\\
&Q_{1}:=\{ M \to N \mid \textnormal{$\mu_{X}(M)=N$ for some minimal $\nu$-stable $X$}\}. \notag
\end{align}
\begin{proposition}
Let $M,N$ be basic $\nu$-stable silting objects of $\mathcal{T}$.
Then the following statements are equivalent:
\begin{itemize}
\item[(1)] $N$ is an irreducible $\nu$-stable mutation of $M$
\item[(2)] $M>N$ and there exists no $L\in \silt^{\nu}\mathcal{T}$ satisfying $M>L>N$.
\end{itemize}
In particular, $Q(\silt^{\nu}\mathcal{T})$ is quiver isomorphic to the Hasse quiver of $(\silt^{\nu}\mathcal{T}, \ge)$.
\end{proposition}
\begin{proof}
The proof is the same as in \cite[Theorem 2.35]{AI12} and \cite[Theorem 5.11]{CKL15}.
\end{proof}

\subsection{$\nu$-stable silting-discrete triangulated category}
In this subsection, we introduce  the notion of $\nu$-stable silting-discrete triangulated categories, which gives a unification of silting-discrete triangulated categories and tilting-discrete bounded homotopy categories of finitely generated projective modules for self-injective algebras.
Recall the definition of silting-discrete triangulated categories which is introduced in \cite{Ai13}.
A triangulated category $\mathcal{T}$ with a silting object is said to be \emph{silting-discrete} if for each $M\in \silt\mathcal{T}$ and $d\in \mathbb{Z}_{\ge 0}$, the set
\begin{align}
\ssilt{(d+1)_{M}}\mathcal{T}:=\{ N\in \silt\mathcal{T} \mid M \ge N \ge M[d]\} \notag
\end{align}
is finite.
Moreover, a finite dimensional algebra is called a \emph{silting-discrete algebra} if the bounded homotopy category $\Kb(\proj A)$ is silting-discrete.
Similarly, we define \emph{tilting-discrete triangulated categories} and \emph{tilting-discrete algebras}.

Now we introduce the notion of $\nu$-stable silting-discrete triangulated categories.

\begin{definition}
\begin{itemize}
\item[(1)] Assume that $\mathcal{T}$ has a $\nu$-stable silting object.
A triangulated category $\mathcal{T}$ is said to be \emph{$\nu$-stable silting-discrete} if for each $M\in \silt^{\nu}\mathcal{T}$ and $d\in \mathbb{Z}_{\ge 0}$, the set
\begin{align}
\ssilt{(d+1)_{M}}^{\nu}\mathcal{T} :=\{ N\in \silt^{\nu}\mathcal{T} \mid M \ge N \ge M[d]\} \notag
\end{align}
is finite. Note that $\ssilt{1_{M}}^{\nu}\mathcal{T}=\{ M \}/\cong$.
\item[(2)] Let $A$ be a finite dimensional algebra and $\nu:\Kb(\proj A)\to \Kb(\proj A)$ a triangle auto-equivalence.
We call $A$ a \emph{$\nu$-stable silting-discrete algebra} if $\Kb(\proj A)$ is $\nu$-stable silting-discrete.
\end{itemize}
\end{definition}

The following example shows that $\nu$-stable silting-discrete triangulated categories unify silting-discrete triangulated categories and tilting-discrete self-injective algebras.

\begin{example}
\begin{itemize}
\item[(1)] Assume that $\nu$ is functorial isomorphic to the identity functor. 
Then $\nu$-stable silting objects are exactly silting objects.
Moreover, $\nu$-stable silting-discrete triangulated categories coincide with silting-discrete triangulated categories.   
\item[(2)] Let $A$ be a self-injective algebra. Then $\Kb(\proj A)$ has a Serre functor $\nu$.
By Proposition \ref{prop:selfinjective-tilting}, all tilting objects are $\nu$-stable silting objects. 
Hence $A$ is $\nu$-stable silting-discrete if and only if it is tilting-discrete.
\end{itemize}
\end{example}

As a generalization of \cite[Proposition 2.14]{AH06} and \cite[Corollary 2.43]{AI12}, we provide an example of $\nu$-stable silting-discrete triangulated categories which plays an important role in the next section.
Let $M=\oplus_{i\in I}M_{i}$ be a basic $\nu$-stable object of $\mathcal{T}$, where each $M_{i}$ is indecomposable.
Define a permutation $v_{M}:I\to I$ as $\nu M_{i}\cong M_{v_{M}(i)}$. 
We call $M$ a \emph{$\nu$-cyclic object} if $v_{M}$ acts transitively on $I$.

\begin{proposition}\label{prop:cyclic-silting}
Let $A$ be a $\nu$-cyclic silting object of $\mathcal{T}$.
Then we have 
\begin{align}
\silt^{\nu}(\mathcal{T})=\{ A[i]\mid i\in \mathbb{Z}\}. \notag
\end{align}
In particular, $\mathcal{T}$ is $\nu$-stable silting-discrete.
\end{proposition}
\begin{proof}
Let $M\in \silt^{\nu}\mathcal{T}$.
By \cite[Proposition 2.23]{AI12}, there exist integers $m_{1}\le m_{2}\in \mathbb{Z}$ such that $M\in \add A[m_{1}]\ast \add A[m_{1}+1]\ast\cdots\ast \add A[m_{2}]$, $M\notin \add A[m_{1}+1]\ast\cdots \ast\add A[m_{2}]$ and $M\notin \add A[m_{1}]\ast\cdots\ast\add A[m_{2}-1]$.
Let $L:=M[-m_{1}]$ and $l:=m_{2}-m_{1}$.
Suppose $l \neq 0$.
By Lemma \ref{lem:int-mutation}, there exist two triangles
\begin{align}
&A'\xto{f} L\to L''\to A'[1],\notag\\
&L'\to L\xto{g} A''[l]\to L'[1]\notag
\end{align}
such that $A',A''\in \add A$ are non-zero, $f$ is a minimal right $\add A$-approximation and $g$ is a minimal left $\add A[l]$-approximation.
Then it follows from \cite[Lemma 2.25]{AI12} that $\add A'\cap \add A''=\{ 0\}$.
On the other hand, by Lemma \ref{lem:isom-triang}, we have $\nu A'\cong A'$ and $\nu A''\cong A''$.
Since $A$ is $\nu$-cyclic, we obtain $\add A'=\add A=\add A''$, a contradiction.
This implies $l=0$ and hence $M=A[m_{1}]$.
\end{proof}

The following proposition is one of nice properties of $\nu$-stable silting-discrete triangulated categories.
\begin{proposition}
Assume that $\mathcal{T}$ is $\nu$-stable silting-discrete. 
If $M,N\in \silt^{\nu}\mathcal{T}$ with $M> N$, then $N$ can be obtained from $M$ by iterated irreducible $\nu$-stable mutation.
\end{proposition}
\begin{proof}
By \cite[Proposition 2.23]{AI12}, there exists an integer $d$ such that $M> N\ge M[d+1]$.
Since $\mathcal{T}$ is $\nu$-stable silting-discrete, the set $\ssilt{d_{M}}^{\nu}\mathcal{T}$ is finite.
Hence the assertion follows from Proposition \ref{prop:local-property-nu-mutation}(2).
\end{proof}
 
Next, following \cite[Theorem 3.8]{Ai13} and \cite[Theorem 2.4]{AM17}, we give a characterization of triangulated categories to be $\nu$-stable silting-discrete.
\begin{theorem}\label{thm:equiv-nu-silting-discrete}
Let $\mathcal{T}$ be a triangulated category with a $\nu$-stable silting object.
Then the following statements are equivalent.
\begin{itemize}
\item[(1)] $\mathcal{T}$ is $\nu$-stable silting-discrete.
\item[(2)] $\mathcal{T}$ admits a $\nu$-stable silting object $A$ such that, for each integer $d\ge 0$, the set $\ssilt{(d+1)_{A}}^{\nu}\mathcal{T}$ is finite.
\item[(3)] Fix any basic $\nu$-stable silting object $A$. For each object $M$ obtained by a finite sequence of irreducible $\nu$-stable mutations from $A$, the set $\ssilt{2_{M}}^{\nu}\mathcal{T}$ is finite. 
\end{itemize}
\end{theorem}

For convenience of readers, we give a proof of Theorem \ref{thm:equiv-nu-silting-discrete}.
We need the following lemma.

\begin{lemma}\label{lem:technical-lem}
Fix a basic $\nu$-stable silting object $A$ of $\mathcal{T}$ and an integer $d\ge 1$.
Let $M\in \ssilt{(d+1)_{A}}^{\nu}\mathcal{T}$.
Then the following statements hold.
\begin{itemize}
\item[(1)] Let $A'\in \ssilt{2_{A}}^{\nu}\mathcal{T}$ with $A'\neq A[1]$ and $A'\ge M \ge A[d]$.
If $M$ is not in $\ssilt{d_{A'}}^{\nu}\mathcal{T}$, then there exists an irreducible $\nu$-stable mutation $A''$ of $A'$ such that $A'>A''\ge \{ M,A[1] \}$.
\item[(2)] If $\ssilt{2_{A}}^{\nu}\mathcal{T}$ is a finite set, then there exists $N\in \ssilt{2_{A}}^{\nu}\mathcal{T}$ such that $M\in \ssilt{d_{N}}^{\nu}\mathcal{T}$.
\end{itemize}
\end{lemma}
\begin{proof}
(1) By our assumption, $M\in \ssilt{(d+1)_{A'}}^{\nu}\mathcal{T}\setminus \ssilt{d_{A'}}^{\nu}\mathcal{T}$.
Then it follows from Lemma \ref{lem:int-mutation} that there exists a triangle 
\begin{align}
M'\to M \xto{f} P'[d]\xto{f'}M'[1] \notag
\end{align}
such that $M'\in \add A'\ast \cdots \ast \add A'[d-1]$, $0\neq P'\in \add A'$, $f$ is a minimal left $\add A'[d]$-approximation of $M$ and $f'$ belongs to the radical of $\mathcal{T}$.
By Lemma \ref{lem:isom-triang}, $P'$ is $\nu$-stable. 
Moreover, by a similar argument, $A[1]\in \ssilt{2_{A'}}^{\nu}\mathcal{T}\setminus \ssilt{1_{A'}}^{\nu}\mathcal{T}$ induces a triangle 
\begin{align}
Q'\xto{g'}R'\xto{g''} A[1]\xto{g} Q'[1]\notag
\end{align}
such that $R',Q'\in \add A'$ are $\nu$-stable, $g$ is a minimal left $\add A'[1]$-approximation of $A[1]$ and $g'$ is in the radical of $\mathcal{T}$.
Since $A$ and $A'$ are silting, so is $Q'\oplus R'$.
This implies $\add A'=\add(Q'\oplus R')$.
On the other hand, it follows from \cite[Lemma 2.25]{AI12} that $\add P'\cap \add R'=\{ 0\}$.
Hence we obtain $P'\in \add Q'$.
Take a minimal $\nu$-stable direct summand $X$ of $P'$.
By Lemma \ref{lem:int-mutation}(2), $A'> \mu_{X}(A') \geq\{ M,A[1]\}$.
Hence, putting $A'':=\mu_{X}(A')$, we have the assertion.

(2) If $M\in \ssilt{d_{A}}^{\nu}\mathcal{T}$, then there is nothing to prove.
In the following, we assume $M\notin \ssilt{d_{A}}^{\nu}\mathcal{T}$.
By (1), there exists an irreducible $\nu$-stable mutation $A'$ of $A$ such that $A>A'\ge \{ M,A[1] \}$.
Hence $A'\in \ssilt{2_{A}}^{\nu}\mathcal{T}$ and $M\in \ssilt{(d+1)_{A'}}^{\nu}\mathcal{T}$.
If $M\in \ssilt{d_{A'}}^{\nu}\mathcal{T}$, then we obtain the desired result.
We assume $M\notin \ssilt{d_{A'}}^{\nu}\mathcal{T}$.
If $A'=A[1]$, then $M\in \ssilt{(d+1)_{A}}^{\nu}\mathcal{T}$ implies $\Hom(M,A[d])=0$ and hence $M\in \ssilt{d_{A'}}^{\nu}\mathcal{T}$ a contradiction.
Therefore we assume $A'\neq A[1]$.
By (1), there exists an irreducible $\nu$-stable mutation $A''$ of $A'$ such that $A'>A''\ge \{ M, A[1]\}$.
Thus, we obtain $A''\in \ssilt{2_{A}}^{\nu}\mathcal{T}$ and $M\in\ssilt{(d+1)_{A''}}^{\nu}\mathcal{T}$.
By repeated use of this argument, we have a sequence of irreducible $\nu$-stable mutations in $\ssilt{2_{A}}^{\nu}\mathcal{T}$.
Since $\ssilt{2_{A}}^{\nu}\mathcal{T}$ is finite, this procedure stops after a finite number of steps.
Therefore, there exists $N\in \ssilt{2_{A}}^{\nu}\mathcal{T}$ such that $M\in \ssilt{d_{N}}^{\nu}\mathcal{T}$.
This finishes the proof.
\end{proof}

Now we are ready to prove Theorem \ref{thm:equiv-nu-silting-discrete}.

\begin{proof}[Proof of Theorem \ref{thm:equiv-nu-silting-discrete}]
(1)$\Rightarrow$(2) and (1)$\Rightarrow$(3) clearly hold.

(2)$\Rightarrow$(1): The proof is the same as in \cite[Proposition 3.8]{Ai13}.

(3)$\Rightarrow$(2): We show that $\ssilt{(d+1)_{A}}^{\nu}\mathcal{T}$ is finite for all $d\ge 0$. 
If $d\le 1$, then this is clear.
Let $d\ge 2$ and $M\in \ssilt{(d+1)_{A}}^{\nu}\mathcal{T}$.
By assumption, the set $\ssilt{2_{A}}^{\nu}\mathcal{T}$ is finite.
Thus it follows from Lemma \ref{lem:technical-lem}(2) that there exists $A_{1}\in \ssilt{2_{A}}^{\nu}\mathcal{T}$ such that $M\in\ssilt{d_{A_{1}}}^{\nu}\mathcal{T}$.
Since $\silt^{\nu}[A_{1},A]$ is a finite set, it follows from Proposition \ref{prop:local-property-nu-mutation}(2) that $A_{1}$ is obtained by a finite sequence of irreducible $\nu$-stable mutations from $A$.
Hence $\ssilt{2_{A_{1}}}^{\nu}\mathcal{T}$ is also a finite set.
By repeated use of this argument, we have
\begin{align}
\displaystyle \ssilt{(d+1)_{A_{0}}}^{\nu}\mathcal{T}\subseteq \bigcup_{A_{1}\in\ssilt{2_{A_{0}}}^{\nu}\mathcal{T}}\bigcup_{A_{2}\in\ssilt{2_{A_{1}}}^{\nu}\mathcal{T}}\cdots\bigcup_{A_{d-1}\in\ssilt{2_{A_{d-2}}}^{\nu}\mathcal{T}}\ssilt{2_{A_{d-1}}}^{\nu}\mathcal{T}, \notag
\end{align}
where $A_{0}:=A$.
Due to the construction, the set $\ssilt{2_{A_{i}}}^{\nu}\mathcal{T}$ is finite for each $i\ge 0$.
This implies that the set $\ssilt{(d+1)_{A}}^{\nu}\mathcal{T}$ is finite.
\end{proof}

We can recover \cite[Theorem 1.2]{AM17}.
\begin{corollary}\label{cor:AM17-thm}
Let $A$ be a finite dimensional algebra $($respectively, a finite dimensional self-injective algebra$)$ and $\nu$ an identity functor $($respectively, a Serre functor$)$.
Then the following conditions are equivalent.
\begin{itemize}
\item[(1)] $\Kb(\proj A)$ is silting-discrete $($respectively, tilting-discrete$)$.
\item[(2)] For each basic silting $($respectively, tilting$)$ object $M$ obtained by finite sequence of irreducible $\nu$-stable mutations from $A$, the set $\ssilt{2_{M}}^{\nu}\Kb(\proj A)$ is finite.
\end{itemize}
\end{corollary}

As an application, we give an example of $\nu$-stable silting-discrete triangulated categories.
\begin{example}
Let $A$ be a representation-finite self-injective algebra over an algebraically closed field and $\nu$ a Serre functor.
Then $\ssilt{2_{A}}^{\nu}\Kb(\proj A)$ is a finite set.
Since the class of representation-finite self-injective algebras is derived invariant, $\Kb(\proj A)$ is $\nu$-stable silting-discrete, and hence tilting-discrete.
\end{example}

The following theorem is one of our main results of this paper.
\begin{theorem}\label{thm:weakly-symmetric-silting-discrete}
Assume that $\mathcal{T}$ admits a weakly symmetric silting object.
\begin{itemize}
\item[(1)] The following statements are equivalent:
\begin{itemize}
\item[(a)] $\mathcal{T}$ is silting-discrete.
\item[(b)] $\mathcal{T}$ is $\nu$-stable silting-discrete.
\end{itemize}
In this case, all silting objects are weakly symmetric.
\item[(2)] Moreover, if $\nu$ is a Serre functor, then the following statement is also equivalent to $(a)$ and $(b)$:
\begin{itemize}
\item[(c)] $\mathcal{T}$ is tilting-discrete.
\end{itemize}
In this case, all silting objects are tilting.
\end{itemize}
\end{theorem}
\begin{proof}
(1) (a)$\Rightarrow$(b) is clear. 
We prove (b)$\Rightarrow$(a).
We show $\silt\mathcal{T}=\silt^{\nu}\mathcal{T}$.
Let $M\in \silt\mathcal{T}$.
Take a basic weakly symmetric silting object $A$ of $\mathcal{T}$.
Due to \cite[Proposition 2.23]{AI12}, there exists an integer $n>0$ such that $A[-n]\ge M\ge A[n]$.
Then $A':=A[-n]$ is clearly a weakly symmetric silting object and $\ssilt{(2n)_{A'}}^{\nu}\mathcal{T}$ is finite by (b).
Therefore it follows from Proposition \ref{prop:local-property-nu-mutation}(2) that $M$ is weakly symmetric.
In particular, $M\in \silt^{\nu}\mathcal{T}$.
Hence we obtain $\silt\mathcal{T}=\silt^{\nu}\mathcal{T}$.
This implies that, for each $d\ge 0$, the set $\ssilt{(d+1)_{A}}\mathcal{T}=\ssilt{(d+1)_{A}}^{\nu}\mathcal{T}$ is a finite set by (b).
By applying Theorem \ref{thm:equiv-nu-silting-discrete} to the case where $\nu=\mathrm{id}$, $\mathcal{T}$ is silting-discrete.

(2) Assume that $\nu$ is a Serre functor.
By Proposition \ref{prop:nu-silting-tilting}, (a)$\Rightarrow$(c)$\Rightarrow$(b) holds.
Hence the assertion follows from (1).
\end{proof}


As an application of Theorem \ref{thm:weakly-symmetric-silting-discrete}, we have the following result.

\begin{corollary}\label{cor:weakly-symmetric-silting-discrete}
Let $A$ be a weakly symmetric algebra.
Then $A$ is silting-discrete if and only if it is tilting-discrete.
\end{corollary}
\begin{proof}
Clearly, $A$ is a weakly symmetric silting object of $\Kb(\proj A)$.
On the other hand, since a weakly symmetric algebra is an Iwanaga--Gorenstein algebra, $\Kb(\proj A)$ admits a Serre functor $\nu$.
Therefore the assertion follows from Theorem \ref{thm:weakly-symmetric-silting-discrete}(2).
\end{proof}

We give concrete examples for Corollary \ref{cor:weakly-symmetric-silting-discrete}.

\begin{example}\label{ex:preprojective}
Let $A$ be the preprojective algebra of one of Dynkin diagrams $\mathbf{D}_{2n}$, $\mathbf{E}_{7}$ and $\mathbf{E}_{8}$.
Then $A$ is weakly symmetric (see \cite{BBK02}) and tilting-discrete by \cite[Theorem 1.3]{AM17}.
By Corollary \ref{cor:weakly-symmetric-silting-discrete}, $A$ is silting-discrete.
\end{example}

Remark that Example \ref{ex:preprojective} can be obtained by \cite[Theorem 1.1]{AM17} because $\Delta=\Delta^{\mathrm{f}}$ holds for $\Delta=\mathbf{D}_{2n}$, $\mathbf{E}_{7}$ or $\mathbf{E}_{8}$.

\section{The first example: trivial tilting-discrete case}
Let $A$ be a basic connected non-semisimple self-injective algebra over an algebraically closed field $K$ and let $\nu=\nu_{A}:=\kD\Hom_{A}(-,A)$ a Nakayama functor. Note that $\nu$ is a Serre functor in $\Kb(\proj A)$.
Our aim of this section is to give a construction of a self-injective algebra $\widetilde{A}$ satisfying the following two properties:
\begin{itemize}
\item $\tilt{\Kb(\proj \widetilde{A})}=\{ A[i]\mid i\in \mathbb{Z}\}$. In particular, $\widetilde{A}$ is tilting-discrete.
\item $\widetilde{A}$ is not silting-discrete.
\end{itemize}

Let $Q=(Q_{0},Q_{1})$ be a finite quiver, where $Q_{0}$ is the vertex set and $Q_{1}$ is the arrow set of $Q$.
We denote by $KQ_{l}$ the subspace of $KQ$ generated by all paths of length $l$.
Define a new quiver $\widetilde{Q}=(\widetilde{Q}_{0},\widetilde{Q}_{1})$ as $\widetilde{Q}_{0}:=Q_{0}$ and $\widetilde{Q}_{1}:=Q_{1}^{+}\coprod Q_{1}^{-}$, where $Q_{1}^{+}:=\{ a^{+}\mid a\in Q_{1}\}$ and $Q_{1}^{-}:=\{ a^{-}\mid a\in Q_{1}\}$.
The correspondences $a\mapsto a^{\pm}$ induce $K$-linear isomorphisms $(-)^{\pm}:KQ_{1}\to KQ_{1}^{\pm}$ and moreover, they are extended to $K$-linear isomorphisms $(-)^{\pm}:\oplus_{l\ge 1}KQ_{l}\to\oplus_{l\ge 1} KQ_{l}^{\pm}$.

Assume that $A=KQ/I$ is self-injective, where $I$ is an admissible ideal of $KQ$.
Let $I^{d}:=\langle a^{+}b^{-}, a^{-}b^{+}\rangle_{K\widetilde{Q}}$ and $I^{c}:=\oplus_{i\in Q_{0}}K(p_{i}^{+}-p_{i}^{-})$, where $p_{i}\notin I$ is a path in $KQ$ with maximal length starting from $i\in Q_{0}$.
Define a subspace $\widetilde{I}$ of $K\widetilde{Q}$ by $\widetilde{I}:=I^{+}+I^{-}+I^{d}+I^{c}$.
Then we have the following lemma.

\begin{lemma}
The following statements hold.
\begin{itemize}
\item[(1)] The subspace $\widetilde{I}$ is a two-sided ideal of $K\widetilde{Q}$.
\item[(2)] If $\soc P(i)\subset \rad^{2}A$ for each $i\in Q_{0}$, then $\widetilde{I}$ is admissible.
\end{itemize}
\end{lemma} 
\begin{proof}
(1) We can easily check that $I^{+}+I^{-}+I^{d}$ is a two-side ideal and $\widetilde{I}$ is a right ideal.
To complete the proof, we show $a^{\pm}(p_{i}^{+}-p_{i}^{-})\in \widetilde{I}$, that is, $a^{+}p_{i}^{+}, a^{-}p_{i}^{-}\in \widetilde{I}$ for all $i\in Q_{0}$ and $a\in Q_{1}$.
Thus it is enough to claim $ap_{i}\in I$. 
Indeed, if it is true, then $a^{\pm}p_{i}^{\pm}=(ap_{i})^{\pm}\in I^{\pm}$.
Suppose to contrary that $ap_{i}\notin I$.
If $a$ is a loop, it contradicts to the maximality of the length of $p_{i}$.  
On the other hand, if $a:h\to i$ is not a loop, then $ap_{i}\in \soc P(h)$.
This implies $\soc P(i)=\soc P(h)$, a contradiction to self-injectivity.

(2) For each $i\in Q_{0}$, the length of $p_{i}$ is at least two by our assumption.
Since $I$ is an admissible ideal, we have the assertion.
\end{proof}

The following theorem is one of main results in this paper.

\begin{theorem}\label{thm:local-tilting}
Assume that $\soc P(i)\subset \rad^{2}A$ for each $i\in Q_{0}$.
Let $\widetilde{A}:=K\widetilde{Q}/\widetilde{I}$. 
Then the following statements hold.
\begin{itemize}
\item[(1)] $\widetilde{A}$ is a basic self-injective algebra. 
\item[(2)] $\widetilde{A}$ is not silting-discrete.
\item[(3)] $A$ is $\nu_{A}$-cyclic object in $\Kb(\proj A)$ if and only if $\widetilde{A}$ is $\nu_{\widetilde{A}}$-cyclic object in $\Kb(\proj \widetilde{A})$. 
\item[(4)] If the equivalence condition in \textup{(3)} is satisfied, then we have 
\begin{align}
\tilt \Kb(\proj \widetilde{A}) =\{ \widetilde{A}[i]\mid i\in \mathbb{Z}\}.\notag
\end{align} 
In particular, $\widetilde{A}$ is a tilting-discrete algebra.
\end{itemize}
\end{theorem}

By \cite[Theorem 3.2]{AIR14}, we can translate results of two-term silting theory into results of $\tau$-tilting theory, and vice versa.
Remark that a finite dimensional algebra $\Lambda$ is $\tau$-tilting finite if and only if $\ssilt{2_{\Lambda}}\Kb(\proj \Lambda)$ is a finite set. Moreover, it follows from \cite[Corollary 1.9]{DIRRT} that if $\Lambda$ is $\tau$-tilting finite, then all factor algebras are also $\tau$-tilting finite.  

\begin{proof}
In the rest of this section, the statements (1) and (3) are shown in a more general setting.

(2) Note that $Q$ is a Dynkin quiver if and only if $KQ$ is $\tau$-tilting finite (for example, see \cite[Theorem 2.6]{Ad16}).
Since $\widetilde{A}$ contains the path algebra of Kronecker type as a factor algebra, the set $\ssilt{2_{\widetilde{A}}}\Kb(\proj \widetilde{A})$ is not finite.
Hence $A$ is not a silting-discrete algebra.

(4) By the assumption, $\widetilde{A}$ is a $\nu_{\widetilde{A}}$-cyclic object in $\Kb(\proj \widetilde{A})$.
Thus the assertion follows from Propositions \ref{prop:selfinjective-tilting} and \ref{prop:cyclic-silting}.
\end{proof}

Before proving Theorem \ref{thm:local-tilting}(1) and (3), we give an example of $\widetilde{A}$.

\begin{example}
Let $A=KQ/I$ be a self-injective Nakayama algebra, where $Q=(\xymatrix{1\ar@<0.5ex>[r]^-{a}&2\ar@<0.5ex>[l]^-{b}})$ and $I=\langle abab,baba \rangle$.
Then $\widetilde{A}=K\widetilde{Q}/\widetilde{I}$ is given by
\begin{align}
\widetilde{Q}=\xymatrix{1\ar@<3ex>[r]^-{a^{-}}\ar@<0.5ex>[r]^-{a^{+}}&2\ar@<0.5ex>[l]^-{b^{+}}\ar@<3ex>[l]^-{b^{-}}} \notag
\end{align}
and 
\begin{align}
\widetilde{I}=\langle a^{+}b^{+}a^{+}b^{+}, b^{+}a^{+}b^{+}a^{+}, a^{-}b^{-}a^{-}b^{-},b^{-}a^{-}b^{-}a^{-}, &a^{+}b^{-},a^{-}b^{+}, b^{+}a^{-},b^{-}a^{+},\notag\\
&a^{+}b^{+}a^{+}-a^{-}b^{-}a^{-}, b^{+}a^{+}b^{+}-b^{-}a^{-}b^{-}\rangle.\notag
\end{align}
Thus we obtain
\begin{align}
A_{A}=\begin{smallmatrix}1\\2\\1\\2\end{smallmatrix}\oplus\begin{smallmatrix}2\\1\\2\\1\end{smallmatrix},\hspace{5mm}
\widetilde{A}_{\widetilde{A}}=\begin{smallmatrix}&1&\\2&&2\\1&&1\\&2&\end{smallmatrix}\oplus\begin{smallmatrix}&2&\\1&&1\\2&&2\\&1&\end{smallmatrix}.\notag
\end{align}
Since $\widetilde{Q}$ contains a multiple arrow, $\widetilde{A}$ is not silting-discrete.
On the other hand, we can easily check that $\widetilde{A}$ is a $\nu_{\widetilde{A}}$-cyclic object of $\Kb(\proj \widetilde{A})$. Hence it is a tilting-discrete algebra.
\end{example}

From now we provide a proof of Theorem \ref{thm:local-tilting}(1) and (3) in more general setting.
We start with recalling some properties of basic self-injective algebras.
For detail, see \cite[Chapter IV]{SY11}.
A basic self-injective $K$-algebra $\Lambda$ admits a non-degenerate associative $K$-bilinear form $(-,-)_{\Lambda}:\Lambda\times \Lambda\to K$.
This property induces an algebra automorphism $\mathrm{v}_{\Lambda}:\Lambda\to \Lambda$ with $(\mathrm{v}_{\Lambda}(a),b)_{\Lambda}=(b,a)_{\Lambda}$ for all $a,b\in \Lambda$. 
We call $\mathrm{v}_{\Lambda}$ a \emph{Nakayama automorphism}.
A subspace $\Gamma$ of $\Lambda$ is said to be \emph{$\mathrm{v}_{\Lambda}$-stable} if $\mathrm{v}_{\Lambda}(\Gamma)= \Gamma$.

Let $J$ be a two-sided ideal of $\Lambda$ and $\Gamma$ a subalgebra with $J\subset \Gamma$.
In our convention, the identity of $\Gamma$ coincides with that of $\Lambda$.
Define vector spaces $J',J''$ by
\begin{align}
&J':=\{ \lambda\in \Lambda \mid (J,\lambda)_{\Lambda}=0\},\notag\\
&J'':=\{ \lambda\in \Lambda\mid (\Gamma,\lambda)_{\Lambda}=0\}. \notag
\end{align}
Since $J$ is a two-sided ideal, we can easily check that $J'$ is a two-sided ideal of $\Lambda$.
The following lemma plays an important role in this section.

\begin{lemma}\label{lem:vect-bilinear}
Let $(-,-):=(-,-)_{\Lambda}$ be a non-degenerate associative $K$-bilinear form and $\mathrm{v}:=\mathrm{v}_{\Lambda}$ the Nakayama automorphism associated with $(-,-)$.
Then the restriction $(-,-)_{\Gamma}:=(-,-)|_{\Gamma\times\Gamma}$ is a $K$-bilinear form. 
Moreover, if $J'\subseteq J$, then the following statements hold.
\begin{itemize}
\item[(1)] $J''$ is a vector subspace of $\Gamma$.
\item[(2)] If $\Gamma$ is $\mathrm{v}$-stable, then so is $J''$ and the following statements are equivalent for $\gamma\in \Gamma$.
\begin{itemize}
\item[(a)] $(\gamma,-)_{\Gamma}=0$.
\item[(b)] $\gamma\in J''$.
\item[(c)] $(-,\gamma)_{\Gamma}=0$.
\end{itemize}
In particular, the induced $K$-bilinear form $(-,-)_{\Gamma/J''}:\Gamma/J''\times\Gamma/J''\to K$ is non-degenerate.
\end{itemize}
\end{lemma}
\begin{proof}
Since the former assertion clearly holds, we show the latter assertion.
In the following, we assume $J'\subseteq J$.

(1) By the assumption, we have $J''\subseteq J'\subseteq J\subseteq \Gamma$ as $K$-vector spaces.

(2) Assume that $\Gamma$ is $\mathrm{v}$-stable.
First, we show that $J''$ is $\mathrm{v}$-stable.
Let $j\in J''$.
Then $(\Gamma,j)=0$.
Since $\Gamma$ is $\mathrm{v}$-stable, we have $(\Gamma, \mathrm{v}^{\pm}(j))=(\mathrm{v}^{\mp}(\Gamma),j)=(\Gamma,j)=0$.
Hence $\mathrm{v}^{\pm}(j)\in J''$.
This implies that $J''$ is $\mathrm{v}$-stable.
Next, we prove that the conditions (a), (b) and (c) are equivalent to each other.

(a)$\Rightarrow$(b): By (a), we have $(\Gamma,\mathrm{v}^{-}(\gamma))=(\gamma,\Gamma)=0$.
Thus $\mathrm{v}^{-}(\gamma)\in J''$.
Since $J''$ is $\mathrm{v}$-stable, we obtain $\gamma\in J''$.

(b)$\Rightarrow$(c): This follows from the definition of $J''$.

(c)$\Rightarrow$(a): Since $\Gamma$ is $\mathrm{v}$-stable, we obtain $(\gamma,\Gamma)_{\Gamma}=(\gamma,\Gamma)=(\mathrm{v}(\Gamma),\gamma)=0$.
\end{proof}

By the lemma above, we can construct a self-injective algebra as follows.

\begin{proposition}\label{prop:const-frob}
Keep the notation in Lemma \ref{lem:vect-bilinear}.
Assume that $\Gamma$ is $\mathrm{v}$-stable and $J'\subseteq J$.
Then the following statements hold.
\begin{itemize}
\item[(1)] $J''$ is a two-sided ideal of $\Gamma$. 
\item[(2)] $\Gamma/J''$ is a Frobenius algebra, and hence a self-injective algebra.
\end{itemize}
\end{proposition}
\begin{proof}
(1) It is enough to show $\gamma j\gamma' \in J''$ for each $j\in J''$ and $\gamma,\gamma'\in \Gamma$.
By the associativity, we obtain $(-,\gamma j)=((-)\gamma,j)$ and 
\begin{align}
(-,j\gamma')=((-)j,\gamma')=(\mathrm{v}(\gamma'),(-)j)=(\mathrm{v}(\gamma')(-),j).\notag
\end{align}
By $j\in J''$, we have $(-,\gamma j)_{\Gamma}=0$ and $(-,j\gamma')_{\Gamma}=0$, where the second equation follows from the fact that $\Gamma$ is $\mathrm{v}$-stable.
Hence $\gamma j, j\gamma'\in J''$.

(2) By Lemma \ref{lem:vect-bilinear}(2), $(-,-)_{\Gamma/J''}$ is a non-degenerate associative $K$-bilinear form.
Hence $\Gamma/J''$ is a Frobenius algebra by \cite{Nak39,Nak41} (see also \cite[Theorem IV.2.1]{SY11}).
\end{proof}

In the following, by using Proposition \ref{prop:const-frob}, we prove Theorem \ref{thm:local-tilting}(1).
For a basic connected non-semisimple self-injective $K$-algebra $A=KQ/I$, let $\Lambda:=A\times A$ and $J:=\rad \Lambda=\rad A\times \rad A$.
Then $\Lambda$ is also a self-injective algebra with $((a_{1},b_{1}),(a_{2},b_{2}))_{\Lambda}=(a_{1},b_{1})_{A}+(a_{2},b_{2})_{A}$ for all $(a_{1},b_{1}),(a_{2},b_{2})\in \Lambda$ and $\mathrm{v}_{\Lambda}=\mathrm{v}_{A}\times \mathrm{v}_{A}$ (for example see \cite[Chapter IV]{SY11}).
Define a subset $\Gamma$ by
\begin{align}
\Gamma:=\{ (a,a')\in \Lambda\mid a-a'\in \rad A \}.\notag
\end{align}
Then $\Gamma$ is a subalgebra of $\Lambda$ with $1_{\Gamma}=1_{\Lambda}$ and $J\subset \Gamma$.
Let $J':=\{ (a,a')\in\Lambda\mid (J,(a,a'))_{\Lambda}=0\}$ and $J'':=\{ (a,a')\in\Lambda\mid (\Gamma,(a,a'))_{\Lambda}=0\}$.
The following lemma induces that $\Gamma$ and $J'$ satisfy the assumption in Proposition \ref{prop:const-frob}.

\begin{lemma}\label{lem:key-thm32-1}
Under the notation above, the following statements hold.
\begin{itemize}
\item[(1)] $\Gamma$ is $\mathrm{v}_{\Lambda}$-stable.
\item[(2)] $J'=\soc \Lambda=\soc A\times \soc A$ and $J''=\{ (s,-s)\mid s\in \soc A \}$. 
\item[(3)] $\Gamma/J''$ is a self-injective algebra.
\item[(4)] $A$ is a $\nu_{A}$-cyclic object in $\Kb(\proj A)$ if and only if $\Gamma/J''$ is a $\nu_{\Gamma/J''}$-cyclic object in $\Kb(\proj (\Gamma/J''))$.
\end{itemize}
\end{lemma}
\begin{proof}
(1) This follows from the fact that $a-a'\in\rad A$ if and only if $\mathrm{v}_{A}(a)-\mathrm{v}_{A}(a')\in\rad A$.

(2) First we show $J'=\soc \Lambda$. 
Let $\lambda\in \Lambda$.
By the associativity of $(-,-)_{\Lambda}$, $\lambda\in J'$ if and only if $(-,r\lambda)_{\Lambda}=0$ for all $r\in\rad\Lambda$.
Since $(-,-)_{\Lambda}$ is non-degenerate, we have a result that $(-,r\lambda)_{\Lambda}=0$ if and only if $r\lambda=0$.
By $\soc \Lambda_{\Lambda}=\soc_{\Lambda}\Lambda$, $\lambda\in J'$ if and only if $\lambda\in \soc \Lambda$.

Next, we show $J''=\{ (s,-s)\mid s\in\soc A\}$.
Let $(s,s')\in J'$.
Then $(s,s')\in J''$ if and only if $((a,a'),(s,s'))_{\Lambda}=0$ for each $(a,a')\in \Gamma$.
By $a-a'\in\rad A$, there exists $r\in \rad A$ such that $a'=a+r$.
Thus we obtain 
\begin{align}
0=((a,a'),(s,s'))_{\Lambda}=(a,s)_{A}+(a',s')_{A}=(a,s+s')_{A}+(r,s')_{A}=(a,s+s')_{A}+(1_{A},rs')_{A}\notag
\end{align}
By $\soc A_{A}=\soc {}_{A}A$, we have $rs'=0$ and hence $(a,s+s')_{A}=0$.  
Thus $(s,s')\in J''$ if and only if $(a,s+s')_{A}=0$ for each $a\in A$.
Since $(-,-)_{A}$ is non-degenerate, $(-,s+s')_{A}=0$ if and only if $s+s'=0$.
Hence $J''=\{ (s,-s)\mid s\in \soc A \}$

(3) By (1), $\Gamma$ is $\mathrm{v}_{\Lambda}$-stable.
Since $A$ is a connected non-semisimple self-injective algebra, we have $\soc A \subset \rad A$.
Hence $J'\subset J$ by (2). Thus the assertion follows from Proposition \ref{prop:const-frob}(2).

(4) Since $K$ is an algebraically closed field, we obtain a result that, for primitive idempotents $e,f\in A$, $eA\cong fA$ if and only if $e-f\in \rad A$.

Let $e_{i}$ be the primitive idempotent of $A$ corresponding to a vertex $i\in Q_{0}$.
Then we have $\nu_{A}(e_{i}A)\cong e_{v_{A}(i)}A$, where $v_{A}:Q_{0}\to Q_{0}$ is a Nakayama permutation of $A$.
On the other hand, we obtain $\nu_{A}(e_{i}A)\cong \mathrm{v}_{A}(e_{i})A$ (see \cite[Corollary IV.3.14]{SY11}).
Hence $\mathrm{v}_{A}(e_{i})A\cong e_{v_{A}(i)}A$.
Thus there exists $r\in \rad A$ such that $v_{A}(e_{i})=e_{v_{A}(i)}+r$.
Since
\begin{align}
\mathrm{v}_{\Gamma/J''}((e_{i},e_{i})+J'')
&=(\mathrm{v}_{A}(e_{i}),\mathrm{v}_{A}(e_{i}))+J''=(e_{v_{A}(i)}+r,e_{v_{A}(i)}+r)+J''\notag\\
&=((e_{v_{A}(i)},e_{v_{A}(i)})+J'')+((r,r)+J''),\notag
\end{align}
we have $\mathrm{v}_{\Gamma/J''}((e_{i},e_{i})+J'')-((e_{v_{A}(i)},e_{v_{A}(i)})+J'')\in \rad (\Gamma/J'')$.
This implies $A$ is $\nu_{A}$-cyclic if and only if $\Gamma/J''$ is $\nu_{\Gamma/J''}$-cyclic.
\end{proof}

Comparing $\widetilde{A}$ and $\Gamma/J''$, we complete the proof of Theorem \ref{thm:local-tilting}.

\begin{proposition}\label{prop:key-thm32}
We have an algebra isomorphism $\varphi:\widetilde{A}\to\Gamma/J''$.
\end{proposition}
\begin{proof}
First, we construct an algebra homomorphism $\psi:K\widetilde{Q}\to \Gamma$ which is surjective.
Decompose the identity $1_{A}$ in $A$ as $1_{A}=\sum_{i\in Q_{0}}e_{i}$, where $e_{i}\in A$ is the primitive idempotent corresponding to a vertex $i\in Q_{0}$.
Let $\psi_{0}:\widetilde{Q}_{0}\to \Gamma$ be the map defined by $\psi_{0}(i):=(e_{i},e_{i})$ for $i\in \widetilde{Q}_{0}$, and $\psi_{1}:\widetilde{Q}_{1}\to \Gamma$ the map defined by $\psi_{1}(\alpha^{+}):=(\alpha,0)$ and $\psi_{1}(\alpha^{-}):=(0,\alpha)$ for $\alpha\in Q_{1}$.
Then we can easily check that
\begin{itemize}
\item[(i)] $1_{\Gamma}=\sum_{i\in\widetilde{Q}_{0}}\psi_{0}(i)$ and $\psi_{0}(i)\psi_{0}(j)=\begin{cases}\psi_{0}(i)&(i=j)\\0&(i\neq j)\end{cases}$,
\item[(ii)] for each $\alpha:i\to j$ in $Q_{1}$, $\psi_{1}(\alpha^{\pm})=\psi_{0}(i)\psi_{1}(\alpha^{\pm})\psi_{0}(j)$.
\end{itemize}
By \cite[Theorem II.1.8]{ASS06}, there exists an algebra homomorphism $\psi:K\widetilde{Q}\to \Gamma$ that extends $\psi_{0}$ and $\psi_{1}$. Note that $\psi$ is surjective.
Composing $\psi$ with the natural surjection $\Gamma\to \Gamma/J''$, we obtain a surjective map $\varphi:K\widetilde{Q}\to \Gamma/J''$.

Next, we show $\varphi(\widetilde{I})=0$.
Since $\psi(I^{\pm})=0$ and $\psi(I^{d})=0$, it is enough to claim that $\psi(p_{i}^{+}-p_{i}^{-})\in J''$ or equivalently $\varphi(p_{i}^{+}-p_{i}^{-})=0$, where $p_{i}\notin I$ is a path in $KQ$ with maximal length starting from $i\in Q_{0}$.
Note that $p_{i}\in \soc A$.
By Lemma \ref{lem:key-thm32-1}(2), we have 
\begin{align}
\psi(p_{i}^{+}-p_{i}^{-})=\psi(p_{i}^{+})-\psi(p_{i}^{-})=(p_{i},-p_{i})\in J''.\notag
\end{align}
This implies $\varphi(p_{i}^{+}-p_{i}^{-})=0$ and hence $\varphi(\widetilde{I})=0$.
Thus we obtain a surjective map $\varphi:\widetilde{A}\to \Gamma/J''$.

Finally, we check $\dim_{K}\widetilde{A}=\dim_{K}(\Gamma/J'')$.
Since $\varphi: \widetilde{A}\to \Gamma/J''$ is surjective, we have only to show $\dim_{K}\widetilde{A}\leq \dim_{K}(\Gamma/J'')$.
Define a $K$-linear map $\varphi':\Gamma\to K\widetilde{Q}$ by $\varphi'(e_{i},e_{i})=e_{i}$ for $e_{i}\in A$ and $\varphi'(r,0)=r^{+}$ and $\varphi'(0,r)=r^{-}$ for $r\in \rad A$.
Since $\varphi'(s,-s)=s^{+}-s^{-}\in I^{c}$ for each $s\in \soc A$, we have $\varphi'(J'')\subset \widetilde{I}$.
Thus we obtain a $K$-linear map $\varphi':\Gamma/J'' \to \widetilde{A}$ which is surjective.
This finishes the proof.
\end{proof}

Now we are ready to prove Theorem \ref{thm:local-tilting}(1) and (3).

\begin{proof}[Proof of Theorem \ref{thm:local-tilting}]
By Proposition \ref{prop:key-thm32}, $\widetilde{A}$ is isomorphic to $\Gamma/J''$.
Hence the statements (1) and (3) follow from Lemma \ref{lem:key-thm32-1}. 
\end{proof}

\section{The second example: non-trivial tilting-discrete case}
In this section, we give other examples of tilting-discrete algebras which are not silting-discrete.
For integers $i\leq j$, let $[i,j]:=\{ i,i+1,\ldots,j-1,j \}$.
Let $n,m$ be positive integers.
Define a quiver $\mathbb{T}_{n,m}:=(\mathbb{T}_{0},\mathbb{T}_{1})$, where $\mathbb{T}_{0}$ is the vertex set and $\mathbb{T}_{1}$ is the arrow set, as follows:
\begin{itemize}
\item $\mathbb{T}_{0}:=\{ (i,r)\mid i\in [1,n],\ r\in \mathbb{Z}/m\mathbb{Z} \}$,
\item $\mathbb{T}_{1}:=\{ a_{i,r}: (i,r)\to (i+1,r)\mid i\in [1,n-1],\ r\in \mathbb{Z}/m\mathbb{Z} \}$\\[3pt]
\hspace{15mm}$\coprod\{ b_{i,r}: (i,r)\to (i-1,r+1) \mid i\in [2,n],\ r\in \mathbb{Z}/m\mathbb{Z} \}$.
\end{itemize}
For example, $\mathbb{T}_{5,5}$ is given by the following quiver:
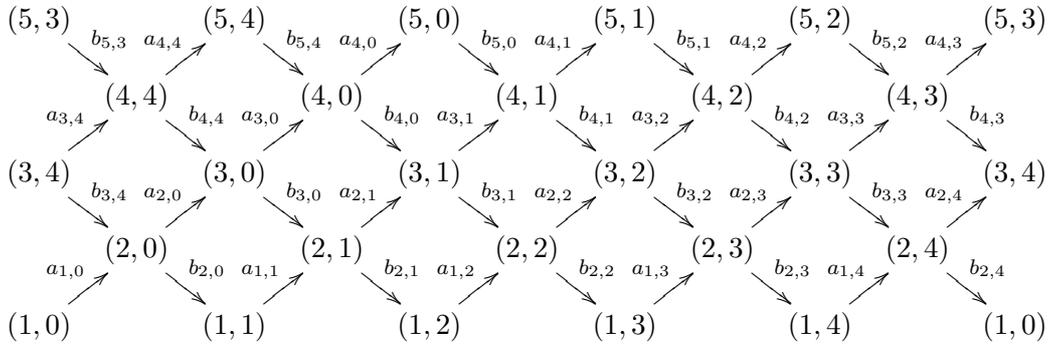
\begin{align}
\xymatrix@C=2mm@R=4mm{
(5,3)\ar[rd]^-{b_{5,3}}&&(5,4)\ar[rd]^-{b_{5,4}}&&(5,0)\ar[rd]^-{b_{5,0}}&&(5,1)\ar[rd]^-{b_{5,1}}&&(5,2)\ar[rd]^-{b_{5,2}}&&(5,3)\\
&(4,4)\ar[ru]^-{a_{4,4}}\ar[rd]^-{b_{4,4}}&&(4,0)\ar[ru]^-{a_{4,0}}\ar[rd]^-{b_{4,0}}&&(4,1)\ar[ru]^-{a_{4,1}}\ar[rd]^-{b_{4,1}}&&(4,2)\ar[ru]^-{a_{4,2}}\ar[rd]^-{b_{4,2}}&&(4,3)\ar[ru]^-{a_{4,3}}\ar[rd]^-{b_{4,3}}\\
(3,4)\ar[ru]^-{a_{3,4}}\ar[rd]^-{b_{3,4}}&&(3,0)\ar[ru]^-{a_{3,0}}\ar[rd]^-{b_{3,0}}&&(3,1)\ar[ru]^-{a_{3,1}}\ar[rd]^-{b_{3,1}}&&(3,2)\ar[ru]^-{a_{3,2}}\ar[rd]^-{b_{3,2}}&&(3,3)\ar[ru]^-{a_{3,3}}\ar[rd]^-{b_{3,3}}&&(3,4)\\
&(2,0)\ar[ru]^-{a_{2,0}}\ar[rd]^-{b_{2,0}}&&(2,1)\ar[ru]^-{a_{2,1}}\ar[rd]^-{b_{2,1}}&&(2,2)\ar[ru]^-{a_{2,2}}\ar[rd]^-{b_{2,2}}&&(2,3)\ar[ru]^-{a_{2,3}}\ar[rd]^-{b_{2,3}}&&(2,4)\ar[ru]^-{a_{2,4}}\ar[rd]^-{b_{2,4}}\\
(1,0)\ar[ru]^-{a_{1,0}}&&(1,1)\ar[ru]^-{a_{1,1}}&&(1,2)\ar[ru]^-{a_{1,2}}&&(1,3)\ar[ru]^-{a_{1,3}}&&(1,4)\ar[ru]^-{a_{1,4}}&&(1,0)
}\notag
\end{align}

We define a self-injective algebra $A_{n,m}$, which plays a crucial role in this section.
Let $K$ be an algebraically closed field.
Formally, put $a_{0,r}=a_{n,r}=b_{1,r}=b_{n+1,r}=0$ for all $r\in \mathbb{Z}/m\mathbb{Z}$.
Then $A_{n,m}$ is a bound quiver algebra $K\mathbb{T}_{n,m}/I$, where $I$ is the two-sided ideal generated by $a_{i,r}b_{i+1,r}-b_{i,r}a_{i-1,r+1}$ for all $i\in[1,n]$ and $r\in \mathbb{Z}/m\mathbb{Z}$.
By definition, $A_{n,1}$ is isomorphic to the preprojective algebra of the Dynkin diagram $\mathbf{A}_{n}$, and in general, $A_{n,m}$ is isomorphic to the stable Auslander algebra of a self-injective Nakayama algebra with $m$ simple modules (up to isomorphism) and Loewy length $n$.
Hence $A_{n,m}$ is a self-injective algebra (see \cite{AR73, Bu98}).

Under certain conditions, the algebra $A_{n,m}$ has the desired property.
Namely, the following theorem is our main result.

\begin{theorem}\label{thm:counterexample-question}
Let $n,m\ge 5$ be integers with $\gcd(n-1,m)=1$. 
Assume that $n$ is odd and $m$ is not divisible by the characteristic of $K$.
Then $A_{n,m}$ is a tilting-discrete algebra but not silting-discrete.
\end{theorem}
Note that $A_{1,1}$ and $A_{2,1}$ are silting-discrete by direct calculation.

Since $A_{n,m}$ is a self-injective algebra, the Nakayama functor $\nu:=\kD\Hom_{A_{n,m}}(-,A_{n,m})$ is a Serre functor in $\Kb(\proj A_{n,m})$.
By Proposition \ref{prop:selfinjective-tilting}, all tilting objects are exactly $\nu$-stable silting objects.
In the following, let $\ssilt{2}A_{n,m}:=\ssilt{2_{A_{n,m}}}\Kb(\proj A_{n,m})$ and $\ttilt{2}{A_{n,m}}:=\ssilt{2_{A_{n,m}}}^{\nu}\Kb(\proj A_{n,m})$.
To show Theorem \ref{thm:counterexample-question}, we need the following two propositions.

\begin{proposition}\label{prop:card-silting-tilting}
Let $n,m$ be positive integers.
Then the following statements hold.
\begin{itemize}
\item[(1)] Assume $n,m\ge 5$. Then $\ssilt{2}A_{n,m}$ is not finite. In particular, $A_{n,m}$ is not silting-discrete.
\item[(2)] Assume that $\gcd(n-1,m)=1$ and $m$ is not divisible by the characteristic of $K$. Then $\ttilt{2}A_{n,m}$ is finite.
\end{itemize}
\end{proposition}

\begin{proposition}\label{prop:derived-class}
Assume that $\gcd(n-1,m)=1$ and $n$ is an odd number.
If $T$ is a tilting object given by iterated irreducible $\nu$-stable mutation from $A_{n,m}$, then the endomorphism algebra $\End_{\Kb(\proj A_{n,m})}(T)$ is isomorphic to $A_{n,m}$. 
\end{proposition}

Before proving the propositions above, we give a proof of Theorem \ref{thm:counterexample-question} by using them.

\begin{proof}[Proof of Theorem \ref{thm:counterexample-question}]
By Proposition \ref{prop:card-silting-tilting}(1), $A_{n,m}$ is not silting-discrete.
We have only to show that $A_{n,m}$ is tilting-discrete. 
Let $T$ be a tilting object given by iterated irreducible $\nu$-stable mutation from $A_{n,m}$ and put $A_{T}:=\End_{\Kb(\proj A_{n,m})}(T)$.
Since $A_{T}$ is isomorphic to $A_{n,m}$ by Proposition \ref{prop:derived-class}, we have $\ttilt{2}A_{T}\cong \ttilt{2}A_{n,m}$. Hence by Proposition \ref{prop:card-silting-tilting}(2), $\ttilt{2}A_{T}$ is finite.
This implies that $A_{n,m}$ is tilting-discrete by Corollary \ref{cor:AM17-thm}.
The proof is complete.
\end{proof}

As an application, we have the following result, which is an analog of \cite[Corollary 1.4]{AM17}.

\begin{corollary}
Let $n,m$ be positive integers with $\gcd(n-1,m)=1$.
Assume that $n$ is odd and $m$ is not divisible by the characteristic of $K$.
For each tilting object $T\in \Kb(\proj A_{n,m})$, the endomorphism algebra $\End_{\Kb(\proj A_{n,m})}(T)$ is Morita equivalent to $A_{n,m}$.
In particular, the derived equivalence class coincides with the Morita equivalence class. 
\end{corollary}
\begin{proof}
By Theorem \ref{thm:counterexample-question}, $A_{n,m}$ is tilting-discrete.
Due to Proposition \ref{prop:local-property-nu-mutation}, each basic tilting object (up to shift) is given by iterated irreducible $\nu$-stable mutation from $A_{n,m}$.
The assertion follows from Proposition \ref{prop:derived-class}.
\end{proof}

In the rest of this section, we prove Propositions \ref{prop:card-silting-tilting} and \ref{prop:derived-class}.
Let $n,m$ be positive integers.
For a vertex $x\in \mathbb{T}_{0}$, let $P(x):=e_{x} A_{n,m}$, $I(x):=\kD(A_{n,m}e_{x})$, and $S(x)=P(x)/\rad P(x)\cong \soc I(x)$.
Formally put $P(0,r)=P(n+1,r)=0$ for all $r\in \mathbb{Z}/m\mathbb{Z}$.
We identify an element of $\Hom_{A_{n,m}}(P(i,r),P(j,s))$ as an element of $e_{(j,s)}A_{n,m}e_{(i,r)}$.
For two-term objects $U=(U_{1}\xto{d_{U}}U_{0})$ and $V=(V_{1}\xto{d_V}V_{0})$ in $\Kb(\proj A_{n,m})$, we denote by $(\varphi_{1},\varphi_{0})$ a morphism in $\Hom_{\Kb(\proj A_{n,m})}(U,V)$, that is, it satisfies the following commutative diagram:
\begin{align}
\xymatrix{
U_{1}\ar[r]^-{d_{U}}\ar[d]^-{\varphi_{1}}&U_{0}\ar[d]^-{\varphi_{0}}\;\\
V_{1}\ar[r]^-{d_{V}}&V_{0}.
}\notag
\end{align}

\subsection{Combinatorial properties of $A_{n,m}$}
In this subsection, we collect combinatorial properties of $A_{n,m}$.
Fix $(i,r)\in \mathbb{T}_0$ and let $w$ be a path starting from $(i,r)$.
For the sake of simplicity, we frequently write down a path without indices, e.g., $a_{i,r}a_{i+1, r}b_{i+2,r}a_{i+1,r+1}=:aaba=:a^2ba$. 
Then we can regard a path $w$ as a word $\mathbf{w}$ with ``$a$'' and ``$b$''. 
We denote by $a(w)$ (respectively, $b(w)$) the number of ``$a$'' (respectively, ``$b$'') in the word $\mathbf{w}$.
Note that $a(w)=b(w)=0$ if and only if $w=e_{i,r}$.
Then, by the definition of the two-sided ideal $I$, we obtain the following properties. 

\begin{lemma}\label{lem:comb-property}
Under the notation above, the following statements hold.
\begin{itemize}
\item[(1)] $w=w'\neq 0$ in $A_{n,m}$ if and only if $\left(a(w),b(w)\right)=\left(a(w'),b(w')\right)\in [0,n-i] \times [0, i-1]$.
\item[(2)] $\{a^{s}b^{t}\mid (s,t)\in [0, n-i] \times [0,i-1]\}$ gives a $K$-basis of $P(i,r)$.
In particular, we have $\dim_{K} P(i,r)=i(n-i+1)$ and $\dim_{K}A_{n,m}=m \dim_{K}A_{n,1}$.
\item[(3)] $P(i,r)\cong I(n-i+1,r+i-1)$. In particular, $\nu P(i,r)\cong P(n-i+1, r+i-n)$.
\end{itemize}
\end{lemma}

\subsection{Proof of Proposition \ref{prop:card-silting-tilting}(1)}

In this subsection, we give a proof of Proposition \ref{prop:card-silting-tilting}(1).

\begin{proof}[Proof of Proposition \ref{prop:card-silting-tilting}(1)]
Assume $n, m\ge 5$. 
Let $e:=e_{4,r-1}+e_{2,r}+e_{3,r}+e_{4,r}+e_{2,r+1}$ for $r\in \mathbb{Z}/m\mathbb{Z}$.
Then $eA_{n,m}e$ is isomorphic to the path algebra of a Euclidean quiver of type $\widetilde{\mathbf{D}}_{4}$.
Hence the set of isomorphism classes of two-term silting objects in $\Kb(\proj eA_{n,m}e)$ is not finite (for example see \cite[Theorem 2.6]{Ad16}).
By \cite[Proposition 2.4]{Ad16} and \cite[Corollary 2.9]{DIJ19}, this implies that $\ssilt{2}A_{n,m}$ is not finite.
This finishes the proof.
\end{proof}

\subsection{$\nu$-stability and $\psi$-stability}
Define an algebra automorphism $\psi:A_{n,m}\to A_{n,m}$ as
\begin{align}
e_{i,r}&\mapsto e_{i,r+1},\notag\\
a_{i,r}&\mapsto a_{i,r+1},\notag\\
b_{i,r}&\mapsto b_{i,r+1}.\notag
\end{align}
Then $\psi$ induces an auto-equivalence $\psi:\mod A_{n,m} \to \mod A_{n,m}$ defined by $\psi(M)a:=M\psi(a)$ for each $M\in\mod A_{n,m}$ and $a\in A_{n,m}$.
By definition, we have $\psi (S(i,r))\cong S(i,r-1)$ and $\psi (P(i,r))\cong P(i,r-1)$.

\begin{lemma}\label{lem:two-stable-silt}
Assume $\gcd(n-1,m)=1$.
Then each $\nu$-stable silting object is $\psi$-stable.
In particular, $\ttilt{2}A_{n,m}$ is a subset of $\ssilt{2}^{\psi}A_{n,m}$.
\end{lemma}
\begin{proof}
For each $(i,r)\in \mathbb{T}_{0}$, we obtain $\nu^{2}P(i,r)\cong \nu P(n-i+1, r+i-n)\cong P(i, r-(n-1))$.
By $\gcd(n-1,m)=1$, there exists an integer $s>0$ such that $\nu^{s}P(i,r)\cong P(i,r-1)\cong\psi (P(i,r))$.
This induces that for any $M\in\ssilt{2}A_{n,m}$, the $g$-vector of $\nu^{s}M$ coincides with that of $\psi (M)$.
By \cite[Theorem 5.5]{AIR14}, we obtain $\nu^{s}M\cong \psi (M)$.
Therefore all $\nu$-stable silting objects are $\psi$-stable.
\end{proof}

\subsection{The algebra $A_{n,m}$ as a skew group algebra}
Recall the definition of skew group algebras. For detail, see \cite{RR85}.
Let $A$ be a finite dimensional $K$-algebra and $\Aut(A)$ the group of $K$-algebra automorphisms of $A$.
Let $G$ be a finite subgroup of $\Aut(A)$ such that the characteristic of $K$ does not divide the order of $G$ (i.e., the group algebra $KG$ is semisimple).
The \emph{skew group algebra} $A\ast G$ is defined as follows: as a $K$-vector space, $A\ast G =A\otimes_{K}KG$ and the multiplication is given by $(a\otimes g)\cdot(a'\otimes g')=a g(a')\otimes gg'$, where $a,a'\in A$ and $g,g'\in G$.
We write $a\ast g$ instead of $a\otimes g$.
Then $A\ast G$ becomes a finite dimensional $K$-algebra with dimension $|G|\dim_{K}A$.

In this subsection, we assume that $m$ is not divisible by the characteristic of $K$.
Let $A_{n}:=A_{n,1}$.
For simplicity, put $e_{i}:=e_{i,0}$, $a_{i}:=a_{i,0}$ and $b_{i}:=b_{i,0}$ in $A_{n}$.
Let $G_{m}=\langle g \rangle$ be a cyclic group of order $m$ acting on $A_{n}$ as 
\begin{align}
&g(e_{i}):= e_{i},\notag \\
&g(a_{i}):=a_{i},\notag\\
&g(b_{i}):=\zeta b_{i},\notag
\end{align}
where $\zeta\in K$ is an $m$-th primitive root of unity.
First we show that there exists a $K$-algebra isomorphism $\varphi: A_{n,m}\to A_{n}\ast G_{m}$.
Define a $K$-algebra homomorphism
\begin{align}
\phi: K\langle e_{i,r}, a_{j,s}, b_{k,t}\mid 1\le i \le n,\ 1\le j\le n-1,\ 2\le k\le n,\ r,s,t\in \mathbb{Z}/m\mathbb{Z} \rangle \to A_{n}\ast G_{m} \notag
\end{align}
by setting
\begin{align}
&e_{i,r}\mapsto \dfrac{1}{m}\sum_{p=0}^{m-1}\zeta^{pr}e_{i}\ast g^{p},\notag\\
&a_{j,s}\mapsto \dfrac{1}{m}\sum_{p=0}^{m-1}\zeta^{ps}a_{j}\ast g^{p},\notag\\
&b_{k,t}\mapsto \dfrac{1}{m}\sum_{p=0}^{m-1}\zeta^{p(t+1)}b_{k}\ast g^{p}.\notag
\end{align}
Put $a_{0,r}=a_{n,r}=b_{1,r}=b_{n+1,r}=0$.
Then we obtain the following equations
\begin{itemize}
\item $\phi(\sum e_{i,r})=1$,
\item $\phi(e_{i,r}e_{i',r'})=\delta_{i,i'}\delta_{r,r'}\phi(e_{i,r})$,
\item $\phi(e_{i,r}a_{j,s})=\delta_{i,j}\delta_{r,s}\phi(a_{j,s})$,
\item $\phi(a_{j,s}e_{i,r})=\delta_{j+1,i}\delta_{s,r}\phi(a_{j,s})$,
\item $\phi(e_{i,r}b_{k,t})=\delta_{i,k}\delta_{r,t}\phi(b_{k,t})$,
\item $\phi(b_{k,t}e_{i,r})=\delta_{k-1,i}\delta_{t+1,r}\phi(b_{k,y})$,
\item $\phi(a_{i,r}b_{i+1,r})=\phi(b_{i,r}a_{i-1,r+1})$,
\end{itemize}
where $\delta$ is the Kronecker delta.
Therefore, $\phi$ induces a $K$-algebra homomorphism $\varphi: A_{n,m}\to A_{n}\ast G_{m}$.
Since
\begin{align}
\left[\begin{smallmatrix}
1 & 1 &\cdots & 1\\
1 & \zeta & \cdots & \zeta^{m-1}\\
1 & \zeta^{2} & \cdots & \zeta^{2(m-1)}\\
\vdots &\vdots  & & \vdots\\
1 & \zeta^{m-1}&\cdots & \zeta^{(m-1)(m-1)}\\
\end{smallmatrix}\right]
\left[\begin{smallmatrix}
e_{i}\ast 1 & a_{j}\ast 1 & b_{k}\ast 1\\
e_{i}\ast g & a_{j}\ast g & b_{k}\ast g\\
e_{i}\ast g^{2} & a_{j}\ast g^{2} & b_{k}\ast g^2\\
\vdots&\vdots&\vdots\\
e_{i}\ast g^{m-1}& a_{j}\ast g^{m-1} & b_{k}\ast g^{m-1}\\
\end{smallmatrix}\right]=m
\left[\begin{smallmatrix}
\varphi(e_{i,0}) & \varphi(a_{j,0}) & \varphi(b_{k,m-1})\\
\varphi(e_{i,1})& \varphi(a_{j,1}) & \varphi(b_{k,0})\\
\varphi(e_{i,2})& \varphi(a_{j,2}) & \varphi(b_{k,1})\\
\vdots&\vdots&\vdots\\
\varphi(e_{i,m-1})& \varphi(a_{j,m-1})& \varphi(b_{k,m-2})
\end{smallmatrix}\right],\notag
\end{align}
we obtain that $\phi$ and $\varphi$ are surjective.
By $\dim_{K}A_{n,m}=\dim_{K}(A_{n}\ast G_{m})$, the map $\varphi$ is an isomorphism.

Next, following \cite{HZ16}, we compare two-term silting objects of $A_{n}$ with those of $A_{n}\ast G_{m}$.
Let $\mathbb{X}$ be the character group of $G_{m}$.
Since $G_{m}$ is a cyclic group, we have $G_{m}\cong\mathbb{X}=\langle \chi \rangle$, where $\chi(g^{s})=\zeta^s$ for each $s\in \mathbb{Z}$.
Then $\mathbb{X}$ acts on $A_{n}\ast G_{m}$ as
\begin{align}
\chi(a \ast g):=\chi(g)a \ast g =\zeta a\ast g \notag
\end{align}
for each $a\in A_{n}$.
It is easy to check that the following diagram commutes:
\begin{align}
\xymatrix{
A_{n,m}\ar[r]^-{\varphi}\ar[d]^-{\psi}&A_{n}\ast G_{m}\ar[d]^-{\chi}\;\\
A_{n,m}\ar[r]^-{\varphi}&A_{n}\ast G_{m}.
}\notag
\end{align}
Since $\chi$ induces an auto-equivalence $\chi$ on $\mod (A_{n}\ast G_{m})$, we have the following lemma.
\begin{lemma}\label{lem:silt-mesh-skewgroup}
Assume that $m$ is not divisible by the characteristic of $K$.
Then there exist isomorphisms
\begin{align}
\ssilt{2}^{g}A_{n}\cong \ssilt{2}^{\chi}(A_{n}\ast G_{m}) \cong \ssilt{2}^{\psi}A_{n,m}. \notag
\end{align}
\end{lemma}
\begin{proof}
Since $G_{m}$ is solvable, we have an isomorphism $\ssilt{2}^{g}A_{n}\to \ssilt{2}^{\chi}(A_{n}\ast G_{m})$ by \cite[Theorem 1.2]{HZ16}.
Moreover, the commutative diagram above induces an isomorphism $\ssilt{2}^{\chi}(A_{n}\ast G_{m})\cong \ssilt{2}^{\psi}A_{n,m}$.
\end{proof}

\subsection{Proof of Proposition \ref{prop:card-silting-tilting}(2)}

Now we are ready to prove Proposition \ref{prop:card-silting-tilting}(2).

\begin{proof}[Proof of Proposition \ref{prop:card-silting-tilting}(2)]
We show that $\ttilt{2}A_{n,m}$ is a finite set.
If $\gcd(n-1,m)=1$, $\ttilt{2}A_{n,m}$ is a subset of $\ssilt{2}^{\psi}A_{n,m}$ by Lemma \ref{lem:two-stable-silt}.
On the other hand, we assume that $m$ is not divisible by the characteristic of $K$.
Then we have $\ssilt{2}^{\psi}A_{n,m}\cong \ssilt{2}^{g}A_{n}$ by Lemma \ref{lem:silt-mesh-skewgroup}.
Since $A_{n}$ is isomorphic to the preprojective algebra of the Dynkin diagram $\mathbf{A}_{n}$, the set $\ssilt{2}A_{n}$ is finite by \cite[Theorem 0.1]{Miz14}.
Hence $\ssilt{2}^{g}A_{n}$ is also a finite set.
This finishes the proof.
\end{proof}

\subsection{Proof of Proposition \ref{prop:derived-class}}
In this subsection, we prove Proposition \ref{prop:derived-class}. 
Assume that $\gcd(n-1,m)=1$ and $n$ is an odd number.
Let $A:=A_{n,m}$ and $\mathcal{T}:=\Kb(\proj A)$.
For each $(\ell,s)\in \mathbb{T}_{0}$, we denote by $\mathcal{O}(\ell,s)$ the $\nu$-orbit of $P(\ell,s)$.
Since $\gcd(n-1,m)=1$, we obtain
\begin{align}
\mathcal{O}_{\ell}:=\mathcal{O}(\ell,s)=\{(\ell,r)\mid r\in \mathbb{Z}/m\mathbb{Z}\}\cup\{(n-\ell+1,r)\mid r\in \mathbb{Z}/m\mathbb{Z}\}. \notag
\end{align}
Without loss of generality, we may assume $\ell \in [1,\frac{n+1}{2}]$.
Then $X_\ell:=\underset{(i,r)\in \mathcal{O}_{\ell}}{\bigoplus} P(i,r)$ is a minimal $\nu$-stable object of $A$.
Thus we have an irreducible $\nu$-stable mutation $\mu_{X_{\ell}}(A)=\underset{(i,r)\in \mathbb{T}_{0}}{\bigoplus}T(i,r)$, where $T(i,r)$ is a two-term object defined as
\begin{align}
T(i,r)= 
\begin{cases}
\overset{\mathrm{-1st}}{P(i,r+1)}\overset{\left[\begin{smallmatrix} a_{i-1,r+1} \\ -b_{i+1,r}\end{smallmatrix}\right]}{\longrightarrow}\overset{\hspace{7mm}\mathrm{0th}}{P(i-1,r+1)\oplus P(i+1,r)} &(i,r)\in \mathcal{O}_{\ell}\\[3pt]
\overset{\mathrm{0th}}{P(i,r)} &(i,r)\notin \mathcal{O}_{\ell}.
\end{cases}\notag
\end{align}
Since the morphism $\left[\begin{smallmatrix} a_{i-1,r+1} \\ -b_{i+1,r}\end{smallmatrix}\right]:P(i,r-1)\to P(i-1,r+1)\oplus P(i+1,r)$ induces a minimal left $\add(A/P(i,r-1))$-approximation in $\mathcal{T}$, we obtain that the mapping cone $T(i,r)$ is indecomposable by Lemma \ref{lem:approximation-result}.

For each $(i,r)\in\mathbb{T}_{0}$, we define two morphisms $x_{i,r}:T(i+1,r)\to T(i,r)$ and $y_{i,r}:T(i-1,r+1)\to T(i,r)$ as follows:
\begin{itemize}
\item If $(i,r)\in \mathcal{O}_{\ell}$, then $x_{i,r}$ and $y_{i,r}$ are given by the following diagrams respectively:
\begin{align}
\xymatrix{
0\ar[r]\ar[d]&P(i+1,r)\ar[d]^-{\left[\begin{smallmatrix} 0 \\ \mathrm{id}\end{smallmatrix}\right]}\\
P(i,r+1)\ar[r]^-{\left[\begin{smallmatrix} a \\ -b\end{smallmatrix}\right]}&{\begin{matrix}P(i-1,r+1)\\\oplus\\ P(i+1,r)
\end{matrix}}}\hspace{15mm}
\xymatrix{
0\ar[r]\ar[d]&P(i-1,r+1)\ar[d]^-{\left[\begin{smallmatrix} \mathrm{id} \\ 0\end{smallmatrix}\right]}\\
P(i,r+1)\ar[r]^-{\left[\begin{smallmatrix} a \\ -b\end{smallmatrix}\right]}&{\begin{matrix}P(i-1,r+1)\\\oplus\\ P(i+1,r)
\end{matrix}}}\notag
\end{align}
\item If $(i+1, r)\in \mathcal{O}_\ell$, then $x_{i,r}$ is given by the following diagram:
\begin{align}
\xymatrix{
P(i+1,r+1)\ar[r]^-{\left[\begin{smallmatrix} a \\ -b\end{smallmatrix}\right]}\ar[d]&{\begin{matrix}P(i,r+1)\\\oplus\\ P(i+2,r)
\end{matrix}}\ar[d]^-{\left[\begin{smallmatrix} a_{i,r}b_{i+1,r} & a_{i,r}a_{i+1,r}\end{smallmatrix}\right]}\\
0\ar[r]&P(i,r)
}\notag
\end{align}
\item If $(i-1, r+1)\in \mathcal{O}_\ell$, then $y_{i,r}$ is given by the following diagram:
\begin{align}
\xymatrix{
P(i-1,r+2)\ar[r]^-{\left[\begin{smallmatrix} a \\ -b\end{smallmatrix}\right]}\ar[d]&{\begin{matrix}P(i-2,r+2)\\\oplus\\ P(i,r+1)\end{matrix}}\ar[d]^-{\left[\begin{smallmatrix} b_{i,r}b_{i-1,r+1}& b_{i,r}a_{i-1,r+1}\end{smallmatrix}\right]}\\
0\ar[r]&P(i,r)
}\notag
\end{align}

\item If otherwise, then $x_{i,r}$ and $y_{i,r}$ are given by the following diagrams respectively:
\begin{align}
\xymatrix{
0\ar[r]\ar[d]&P(i+1,r)\ar[d]^-{a_{i,r}}\\
0\ar[r]&P(i,r)}
\hspace{15mm}
\xymatrix{
0\ar[r]\ar[d]&P(i-1,r+1)\ar[d]^-{b_{i,r}}\\
0\ar[r]&P(i,r)}
\notag
\end{align}
\end{itemize}
Note that $x_{i,r}\in \rad_{\mathcal{T}}(T(i+1,r),T(i,r))$ and $y_{i,r}\in \rad_{\mathcal{T}}(T(i-1,r+1),T(i,r))$.
Moreover, if $i\in [1,n-1]$ (respectively, $i\in [2,n]$), then $x_{i,r}$ (respectively, $y_{i,r}$) is non-zero in $\mathcal{T}$.

We collect properties of two consecutive morphisms.
\begin{lemma}\label{lem:con-rule}
Fix a vertex $(i,r)\in \mathbb{T}_{0}$. 
Then the following equalities hold.
\begin{itemize}
\item[(1)] For $i\in [1, n-2]$, we have
\begin{align}
x_{i,r}x_{i+1,r}=
\begin{cases}
(0, \left[\begin{smallmatrix}0&0\\a_{i+1,r}b_{i+2,r}&a_{i+1,r}a_{i+2,r}\end{smallmatrix}\right])&((i,r)\in\mathcal{O}_{\ell}, (i+2,r)\in\mathcal{O_{\ell}})\\
(0, \left[\begin{smallmatrix}0\\a_{i+1,r}\end{smallmatrix}\right])&((i,r)\in\mathcal{O}_{\ell}, (i+2,r)\notin\mathcal{O_{\ell}})\\
(0, \left[\begin{smallmatrix}a_{i,r}a_{i+1,r}b_{i+2,r}&a_{i,r}a_{i+1,r}a_{i+2,r}\end{smallmatrix}\right])&((i,r)\notin\mathcal{O}_{\ell}, (i+2,r)\in\mathcal{O_{\ell}})\\
(0, a_{i,r}a_{i+1,r})&((i,r)\notin\mathcal{O}_{\ell}, (i+2,r)\notin\mathcal{O_{\ell}}).
\end{cases}\notag
\end{align}
\item[(2)] For $i\in [3, n]$, we have
\begin{align}
y_{i,r}y_{i-1,r+1}=
\begin{cases}
(0, \left[\begin{smallmatrix}b_{i-1,r+1}b_{i-2,r+2}&b_{i-1,r+1}a_{i-2,r+2}\\0&0\end{smallmatrix}\right])&((i,r)\in\mathcal{O}_{\ell}, (i-2,r+2)\in\mathcal{O_{\ell}})\\[5pt]
(0, \left[\begin{smallmatrix}b_{i-1,r+1}\\0\end{smallmatrix}\right])&((i,r)\in\mathcal{O}_{\ell}, (i-2,r+2)\notin\mathcal{O_{\ell}})\\
(0, \left[\begin{smallmatrix}b_{i,r}b_{i-1,r+1}b_{i-2,r+2}&b_{i,r}b_{i-1,r+1}a_{i-2,r+2}\end{smallmatrix}\right])&((i,r)\notin\mathcal{O}_{\ell}, (i-2,r+2)\in\mathcal{O_{\ell}})\\
(0, b_{i,r}b_{i-1,r+1})&((i,r)\notin\mathcal{O}_{\ell}, (i-2,r+2)\notin\mathcal{O_{\ell}}).
\end{cases}\notag
\end{align}
\item[(3)] For each $1\le i \le n-1$, we have
\begin{align}
x_{i,r}y_{i+1,r}=
\begin{cases}
(0, a_{i,r}b_{i+1,r})&((i,r)\notin\mathcal{O}_{\ell})\\
(0,\left[\begin{smallmatrix}
0&0\\
b_{i+1,r}b_{i,r+1}&b_{i+1,r}a_{i,r+1}
\end{smallmatrix}\right])&((i,r)\in\mathcal{O}_{\ell}).
\end{cases}\notag
\end{align}
In particular, if $i=1$, then $x_{i,r}y_{i+1,r}=0$.
\item[(4)] For each $2\le i\le n$, we have
\begin{align}
y_{i,r}x_{i-1,r+1}=
\begin{cases}
(0, b_{i,r}a_{i-1,r+1})&((i,r)\notin \mathcal{O}_{\ell})\\
(0, \left[\begin{smallmatrix}
a_{i-1,r+1}b_{i,r+1}&a_{i-1,r+1}a_{i,r+1}\\
0&0
\end{smallmatrix}\right])&((i,r)\in \mathcal{O}_{\ell}).
\end{cases}\notag
\end{align}
In particular, if $i=n$, then $y_{i,r}x_{i-1,r+1}=0$.
\end{itemize}
\end{lemma}
\begin{proof}
This follows from direct calculations.
\end{proof}

Combining Lemma \ref{lem:con-rule}(3) and (4), we have the following result which will induce commutative relations in $\End_{\mathcal{T}}(\mu_{X_{\ell}}(A))$. 

\begin{lemma}\label{lem:rel-rule}
For each $(i,r)\in \mathbb{T}_{0}$, we have $x_{i,r}y_{i+1,r}-y_{i,r}x_{i-1,r+1}=0$ in $\mathcal{T}$, where $x_{0,r}=x_{n,r}=y_{1,r}=y_{n+1,r}=0$ for all $r\in \mathbb{Z}/m\mathbb{Z}$.
\end{lemma}
\begin{proof}
By Lemma \ref{lem:con-rule}(3) and (4), we have
\begin{align}
x_{i,r}y_{i+1,r}-y_{i,r}x_{i-1,r+1}=
\begin{cases}
(0,a_{i,r}b_{i+1,r}-b_{i,r}a_{i-1,r+1})&((i,r)\notin\mathcal{O}_{\ell})\;\\
\left(0, \left[\begin{smallmatrix} -a_{i-1,r+1}b_{i,r+1} & -a_{i-1, r+1}a_{i,r+1}\\ b_{i+1,r}b_{i,r+1}& b_{i+1,r}a_{i,r+1}\end{smallmatrix}\right]\right)&((i,r)\in\mathcal{O}_{\ell}).
\end{cases}\notag
\end{align}
In the case where $(i,r)\notin\mathcal{O}_{\ell}$, we obtain $x_{i,r}y_{i-1,r+1}-y_{i,r}x_{i-1,r+1}=0$ because $a_{i,r}b_{i+1,r}-b_{i,r}a_{i-1,r+1}=0$ in $A$.
On the other hand, we assume $(i,r)\in\mathcal{O}_{\ell}$.
By definition, we have
\begin{align}
&\begin{bmatrix}
-b_{i,r+1} & -a_{i,r+1}
\end{bmatrix}
\begin{bmatrix}
a_{i-1,r+2}\\-b_{i+1,r+1}
\end{bmatrix}
=0,\notag\\
&\begin{bmatrix}
a_{i-1,r+1}\\ -b_{i+1,r}
\end{bmatrix}
\begin{bmatrix}
-b_{i,r+1} & -a_{i,r+1}
\end{bmatrix}
=\begin{bmatrix} -a_{i-1,r+1}b_{i,r+1} & -a_{i-1, r+1}a_{i,r+1}\\ b_{i+1,r}b_{i,r+1}& b_{i+1,r}a_{i,r+1}\end{bmatrix}.\notag
\end{align}
Hence $x_{i,r}y_{i+1,r}-y_{i,r}x_{i-1,r+1}=0$ in $\mathcal{T}$.
\end{proof}

From now on, we often regard $\Hom_{\mathcal{T}}(T(i,r),T(j,s))$ as a subset of $\End_{\mathcal{T}}(\mu_{X_{\ell}}(A))$ by natural way.
Then we define subsets $\mathsf{X}, \mathsf{Y}$ of $\End_{\mathcal{T}}(\mu_{X_{\ell}}(A))$ as
\begin{align}
\mathsf{X}:=\{x_{i,r}\mid i\in [1,n-1],\ r\in \mathbb{Z}/m\mathbb{Z} \},\ \mathsf{Y}:=\{y_{i,r}\mid i\in [2,n],\  r\in \mathbb{Z}/m\mathbb{Z} \},\notag 
\end{align}
and let $\mathsf{Z}:=\mathsf{X}\coprod \mathsf{Y}$.

In order to determine the Gabriel quiver of $\End_{\mathcal{T}}(\mu_{X_{\ell}}(A))$, we need the following lemma.

\begin{lemma}\label{lem:local-str}
The following statements hold.
\begin{itemize}
\item[(1)] We have 
\begin{align}
\rad_{\mathcal{T}}(T(i,r), T(j,s))=&\Hom_{\mathcal{T}}(T(i-1,r),T(j,s))x_{i-1,r}\notag\\
&+\Hom_{\mathcal{T}}(T(i+1,r-1),T(j,s))y_{i+1,r-1}.\notag
\end{align}
In particular, $\mathsf{Z}$ generates $\End_{\mathcal{T}}(\mu_{X_{\ell}}(A))$ as a $K$-algebra.
\item[(2)] We have $x_{i-1,r}\notin \Hom_{\mathcal{T}}(T(i+1,r-1),T(i-1,r))y_{i+1,r-1}$.
Moreover,
\begin{align}
x_{i-1,r}\in \rad_{\mathcal{T}}(T(i,r),T(i-1,r))\setminus \rad^{2}_{\mathcal{T}}(T(i,r),T(i-1,r)).\notag
\end{align}
\item[(3)] We have $y_{i+1,r-1}\notin \Hom_{\mathcal{T}}(T(i-1,r),T(i+1,r-1))x_{i-1,r}$.
Moreover,
\begin{align}
y_{i+1,r-1}\in \rad_{\mathcal{T}}(T(i,r),T(i+1,r-1))\setminus \rad^{2}_{\mathcal{T}}(T(i,r),T(i+1,r-1)).\notag
\end{align}
\item[(4)] We have 
\begin{align}
\frac{\rad_{\mathcal{T}}(T(i,r),T(j,s))}{\rad^{2}_{\mathcal{T}}(T(i,r),T(j,s))}=
\begin{cases}
\langle \overline{x_{j,s}} \rangle_{K} &\textnormal{if $(j,s)=(i-1,r)$}\\
\langle \overline{y_{j,s}} \rangle_{K} &\textnormal{if $(j,s)=(i+1,r-1)$}\\
0 & \textnormal{if otherwise}.
\end{cases}\notag
\end{align}
\end{itemize}
\end{lemma}
\begin{proof}
(1) Since $x_{i-1,r}$ and $y_{i+1,r-1}$ are in the radical of $\mathcal{T}$, it is enough to show that for each non-isomorphic morphism $\varphi: T(i,r)\to T(j,s)$, there exist morphisms $g':T(i-1,r)\to T(j,s)$ and $g'':T(i+1,r-1)\to T(j,s)$ such that $\varphi=g'x_{i-1,r}+g''y_{i+1,r-1}$.
Consider the following four cases:
\begin{itemize}
\item[(a)] $(i,r)\in \mathcal{O}_\ell$, $(j,s)\not\in \mathcal{O}_\ell$.
\item[(b)] $(i,r)\in \mathcal{O}_\ell$, $(j,s)\in \mathcal{O}_\ell$.
\item[(c)] $(i,r)\not\in \mathcal{O}_\ell$, $(j,s)\in \mathcal{O}_\ell$.
\item[(d)] $(i,r)\not\in \mathcal{O}_\ell$, $(j,s)\not\in \mathcal{O}_\ell$.
\end{itemize}

\underline{Case (a)}: Assume $\varphi=(0,\left[\begin{smallmatrix} f_{1} & f_{2} \end{smallmatrix}\right])$ with $f_{1}: P(i-1,r+1)\to P(j,s)$ and $f_{2}: P(i+1,r)\to P(j,s)$.
By commutative relations of $A$, we can write $f_{1}=g_{1}b_{i,r}+\mathbf{a}$ and $f_{2}=g_{2}a_{i,r}+\mathbf{b}$, where $g_{1},g_{2}\in \Hom_{A}(P(i,r),P(j,s))$, $\mathbf{a}\in\langle a^{k}\mid k\in \mathbb{Z}\rangle_{K}$ and $\mathbf{b}\in \langle b^{k}\mid k\in\mathbb{Z}\rangle_{K}$.
Since 
\begin{align}
0=\begin{bmatrix} f_{1}&f_{2}\end{bmatrix}\begin{bmatrix} a_{i-1,r+1}\\-b_{i+1,r}\end{bmatrix}=f_{1}a_{i-1,r+1}-f_{2}b_{i+1,r}, \notag
\end{align}
we have $(g_{1}-g_{2})a_{i,r}b_{i+1,r}=0$, $\mathbf{a}=0$ and $\mathbf{b}=0$.
Assume $g_{1}-g_{2}\neq 0$ and let $g_{1}-g_{2}=ha^{p}b^{q}$, where $h$ is an invertible element in $e_{j,s}Ae_{j,s}$.
Then we have $ha^{p}b^{q}a_{i,r}b_{i+1,r}$=0, and hence $a^{p+1}b^{q+1}=0$.
By Lemma \ref{lem:comb-property}(1), we obtain $p=n-j$ or $q=j-1$.
If $p=n-j$, then $f_{2}=g_{2}a_{i,r}=(g_{1}-ha^{p}b^{q})a_{i,r}=g_{1}a_{i,r}-ha^{p+1}b^{q}=g_{1}a_{i,r}$.
Since we can take $g_{1}$ as $g_{1}=g'_{1}b_{i+1,r-1}+g''_{1}a_{i-1,r}$ for some $g'_{1}\in \Hom_{A}(P(i+1,r-1),P(j,s))$ and $g''_{1}\in \Hom_{A}(P(i-1,r),P(j,s))$, we have 
$\left[\begin{smallmatrix} f_{1}&f_{2}\end{smallmatrix}\right]=\left[\begin{smallmatrix}g'_{1}&g''_{1}\end{smallmatrix}\right]\left[\begin{smallmatrix}b_{i+1,r-1}b_{i,r}&b_{i+1,r-1}a_{i,r}\\a_{i-1,r}b_{i,r}&a_{i-1,r}a_{i,r}\end{smallmatrix}\right]$.
Hence $\varphi=(0,g''_{1})x_{i-1,r}+(0,g'_{1})y_{i+1,r-1}$.
For the remain cases, we have the assertion by a similar argument. 

\underline{Case (b)}: Assume $\varphi=(\varphi_{1},\varphi_{0})$ with
\begin{align}
\varphi_{0}=\begin{bmatrix}
f_{11} & f_{12}\\
f_{21} & f_{22}\\
\end{bmatrix},\notag
\end{align}
where $f_{11}:P(i-1,r+1)\to P(j-1,s+1)$, $f_{12}:P(i+1,r)\to P(j-1,s+1)$, $f_{21}:P(i-1,r+1)\to P(j+1,s)$, and $f_{22}:P(i+1,r)\to P(j+1,s)$.
Then we can write $\varphi_{1}$ as $\varphi_{1}=g'a_{i-1,r+1}+g''b_{i+1,r}$ for some $g'\in \Hom_{A}(P(i-1,r+1),P(j,s+1))$ and $g''\in \Hom_{A}(P(i+1,r),P(j,s+1))$.
Let $\varphi'_{0}:=\varphi_{0}-\left[\begin{smallmatrix}a_{j-1,s+1}\\-b_{j+1,s}\end{smallmatrix}\right]\left[\begin{smallmatrix}g'&-g''\end{smallmatrix}\right]$.
Since $\varphi'_{0}\left[\begin{smallmatrix}a_{i-1,r+1}\\-b_{i+1,r}\end{smallmatrix}\right]=0$ holds, we have a morphism $(0,\varphi'_{0})$ in $\mathcal{T}$.
By direct calculation, $(\varphi_{1},\varphi_{0})$ is homotopic to $(0,\varphi'_{0})$.
Hence we may always assume $\varphi_{1}=0$.
Therefore we have 
\begin{align}
\varphi=y_{j,s}(0, \left[\begin{smallmatrix}f_{11}&f_{12}\end{smallmatrix}\right])+x_{j,s}(0, \left[\begin{smallmatrix}f_{21}&f_{22}\end{smallmatrix}\right]).\notag
\end{align}
By $(j+1,s),(j-1,s+1)\notin\mathcal{O}_{\ell}$, the assertion follows from the case (a).

\underline{Case (c)}: Assume $\varphi =\left(0,\left[\begin{smallmatrix}f_{1}\\f_{2}\end{smallmatrix}\right]\right)$ with $f_{1}:P(i,r)\to P(j-1,s+1)$ and $f_{2}:P(i,r)\to P(j+1,s)$.
We may assume $f_{1}=a^{p}b^{q}\neq 0$ and $f_{2}=a^{p'}b^{q'}\neq 0$ for some $p,q,p',q'\ge 0$.

If $(i+1,r-1),(i-1,r)\in \mathcal{O}_{\ell}$, then we have $\left(0,\left[\begin{smallmatrix}f_{1}\\0\end{smallmatrix}\right]\right)=\psi_{1}y_{i+1,r-1}$ and $\left(0,\left[\begin{smallmatrix}0\\f_{2}\end{smallmatrix}\right]\right)=\psi_{2}x_{i-1,r}$, where $\psi_{1}:T(i+1,r-1)\to T(j,s)$ and $\psi_{2}:T(i-1,r)\to T(j,s)$ are given by the following diagrams:
\begin{align}
\xymatrix{
P(i+1,r)\ar[r]^-{\left[\begin{smallmatrix} a \\ -b\end{smallmatrix}\right]}\ar[d]^-{a^{p}b^{q}}&{\begin{matrix}P(i,r)\\\oplus\\ P(i+2,r-1)\end{matrix}}\ar[d]^-{\left[\begin{smallmatrix} f_{1}&0 \\ 0&a^{p}b^{q}\end{smallmatrix}\right]}\\
P(j,s+1)\ar[r]^-{\left[\begin{smallmatrix} a \\ -b\end{smallmatrix}\right]}&{\begin{matrix}P(j-1,s+1)\\\oplus\\ P(j+1,s)\end{matrix}}
}\hspace{10mm}
\xymatrix{
P(i-1,r+1)\ar[r]^-{\left[\begin{smallmatrix} a \\ -b\end{smallmatrix}\right]}\ar[d]^-{a^{p'}b^{q'}}&{\begin{matrix}P(i-2,r+1)\\\oplus\\ P(i,r)\end{matrix}}\ar[d]^-{\left[\begin{smallmatrix} a^{p'}b^{q'}&0 \\ 0&f_{2}\end{smallmatrix}\right]}\\
P(j,s+1)\ar[r]^-{\left[\begin{smallmatrix} a \\ -b\end{smallmatrix}\right]}&{\begin{matrix}P(j-1,s+1)\\\oplus\\ P(j+1,s)\end{matrix}}
}\notag
\end{align}

If $(i+1,r-1)\in \mathcal{O}_{\ell}$ and $(i-1,r)\notin \mathcal{O}_{\ell}$, then we have $\left(0,\left[\begin{smallmatrix}f_{1}\\0\end{smallmatrix}\right]\right)=\psi_{1}y_{i+1,r-1}$ and 
\begin{align}
\left(0,\left[\begin{smallmatrix}0\\f_{2}\end{smallmatrix}\right]\right)=
\begin{cases}
\ \left(0,\left[\begin{smallmatrix}0\\a^{p'-1}b^{q'}\end{smallmatrix}\right]\right)x_{i-1,r} &(p'>0)\\
\ \psi_{3}y_{i+1,r-1} &(p'=0),
\end{cases}\notag
\end{align}
where $\psi_{3}$ is given by the following morphism:
\begin{align}
\xymatrix{
P(i+1,r)\ar[r]^-{\left[\begin{smallmatrix} a \\ -b\end{smallmatrix}\right]}\ar[d]^-{-ab^{q'-1}}&{\begin{matrix}P(i,r)\\\oplus\\ P(i+2,r-1)\end{matrix}}\ar[d]^-{\left[\begin{smallmatrix} 0&a^{2}b^{q'-2} \\ f_{2}&0\end{smallmatrix}\right]}\\
P(j,s+1)\ar[r]^-{\left[\begin{smallmatrix} a \\ -b\end{smallmatrix}\right]}&{\begin{matrix}P(j-1,s+1)\\\oplus\\ P(j+1,s)\end{matrix}}
}\notag
\end{align}
Note that if $p'=0$, then $q'\ge 2$ by $(j,s+1),(i+1,r-1)\in \mathcal{O}_{\ell}$. 
For $(i+1,r-1)\notin\mathcal{O}_{\ell}$ and $(i-1,r)\in \mathcal{O}_{\ell}$, we obtain the assertion by a similar argument.

If $(i+1,r-1), (i-1,r)\notin \mathcal{O}_{\ell}$, then we have 
\begin{align}
\left(0, \left[\begin{smallmatrix}f_{1}\\0\end{smallmatrix}\right]\right)=
\begin{cases}
\left(0, \left[\begin{smallmatrix}a^{p-1}\\0\end{smallmatrix}\right]\right)x_{i-1,r}& (q= 0)\\
\left(0, \left[\begin{smallmatrix}a^{p}b^{q-1}\\0\end{smallmatrix}\right]\right)y_{i+1,r-1}& (q\neq 0)
\end{cases}\notag
\end{align}
and 
\begin{align}
\left(0, \left[\begin{smallmatrix}0\\f_{2}\end{smallmatrix}\right]\right)=
\begin{cases}
\left(0, \left[\begin{smallmatrix}0\\a^{p'-1}\end{smallmatrix}\right]\right)x_{i-1,r}& (q'= 0)\\[5pt]
\left(0, \left[\begin{smallmatrix}0\\a^{p'}b^{q'-1}\end{smallmatrix}\right]\right)y_{i+1,r-1}& (q'\neq 0)
\end{cases}\notag
\end{align}

\underline{Case (d)}: Assume $\varphi =(0,f)$ with non-zero $f:P(i,r)\to P(j,s)$. 
We may assume $f=a^{p}b^{q}$ for some $p,q\ge 0$.
If $q\neq 0$, then we have 
\begin{align}
\varphi=
\begin{cases}
(0,a^{p}b^{q-1})y_{i+1,r-1} & ((i+1,r-1)\notin \mathcal{O}_{\ell})\\
(0,[\begin{smallmatrix}f&a^{p+1}b^{q-1}\end{smallmatrix}])y_{i+1,r-1} &((i+1,r-1)\in \mathcal{O}_{\ell}).
\end{cases}\notag
\end{align}
On the other hand, if $q=0$, then we have
\begin{align}
\varphi=
\begin{cases}
(0,a^{p-1})x_{i-1,r} &((i-1,r)\notin \mathcal{O}_{\ell})\\
(0,[\begin{smallmatrix}a^{p-1}b&f\end{smallmatrix}])x_{i-1,r} &((i-1,r)\in \mathcal{O}_{\ell}).
\end{cases}\notag
\end{align}

Therefore we have the assertion.

(2) First we show $x_{i-1,r}\notin \Hom_{\mathcal{T}}(T(i+1,r-1),T(i-1,r))y_{i+1,r-1}$.
Suppose to contrary that $x_{i-1,r}\in \Hom_{\mathcal{T}}(T(i+1,r-1),T(i-1,r))y_{i+1,r-1}$.
Then there exists a non-isomorphic morphism $f:T(i+1,r-1)\to T(i-1,r)$ such that $x_{i-1,r}=fy_{i+1,r-1}$.
By repeated use of (1) and Lemma \ref{lem:rel-rule}, we can write $f=gx_{i,r-1}+\mathbf{y}$ with $g\in \Hom_{\mathcal{T}}(T(i,r-1),T(i-1,r))$ and $\mathbf{y}\in \langle \mathsf{Y} \rangle_{K}$.
Comparing the domain and codomain of $f$, we have $\mathbf{y}=0$.
By Lemma \ref{lem:rel-rule}, 
\begin{align}
x_{i-1,r}=fy_{i+1,r-1}=gx_{i,r-1}y_{i+1,r-1}=gy_{i,r-1}x_{i-1,r}.\notag
\end{align}
This implies $(\mathrm{id}-gy_{i,r-1})x_{i-1,r}=0$. 
Since $\mathrm{id}-gy_{i,r-1}\in \End_{\mathcal{T}}(T(i-1,r))$ is invertible, we have $x_{i-1,r}=0$, a contradiction.

Next we show $x_{i-1,r}\in \rad_{\mathcal{T}}(T(i,r),T(i-1,r))\setminus \rad^{2}_{\mathcal{T}}(T(i,r),T(i-1,r))$.
Suppose to the contrary that $x_{i-1,r}\in \rad_{\mathcal{T}}^{2}(T(i,r),T(i-1,r))$.
We can write $x_{i-1,r}=\sum_{k}f_{k}g_{k}$ with $f_{k}, g_{k}$ radical morphisms.
By (1), we obtain $g_{k}=g'_{k}x_{i-1,r}+g''_{k}y_{i+1,r-1}$ for some morphisms $g'_{k}, g''_{k}$.
Thus $(\mathrm{id}-\sum_{k}f_{k}g'_{k})x_{i-1,r}=\sum_{k}f_{k}g''_{k}y_{i+1,r-1}$.
Since $f_{k}g'_{k}$ is a radical morphism, we have $x_{i-1,r}\in \Hom_{\mathcal{T}}(T(i+1,r-1), T(i-1,r))y_{i+1,r-1}$, a contradiction. 

(3) By an argument similar to (2), we have the assertion.

(4) Let $f\in \rad_{\mathcal{T}}(T(i,r),T(j,s))$. 
By (1), we have $f=f'x_{i-1,r}+f''y_{i+1,r-1}$ for some $f'\in \Hom_{\mathcal{T}}(T(i-1,r),T(j,s))$ and $f''\in \Hom_{\mathcal{T}}(T(i+1,r-1),T(j,s))$.
Note that $x_{i-1,r}$ and $y_{i+1,r-1}$ belong to the radical of $\mathcal{T}$.
If $(j,s)$ is neither $(i-1,r)$ nor $(i+1,r-1)$, then $f'$ and $f''$ are in the radical of $\mathcal{T}$.
Hence $f\in \rad_{\mathcal{T}}^{2}(T(i,r),T(j,s))$.
Assume $(j,s)=(i-1,r)$. Then $f''\in \rad_{\mathcal{T}}(T(i+1,r-1),T(j,s))$. 
Hence the assertion follows from (2).
For $(j,s)=(i+1,r-1)$, the proof is similar.
\end{proof}

By Lemma \ref{lem:local-str}(4), the Gabriel quiver of $\End_{\mathcal{T}}(\mu_{X_{\ell}}(A))$ is isomorphic to $\mathbb{T}_{n,m}$.
Hence there exists a surjective map $\Phi: K\mathbb{T}_{n,m}\to \End_{\mathcal{T}}(\mu_{X_{\ell}}(A))$ given by
\begin{align}
(i,r)&\mapsto T(i,r)\notag\\
a_{i,r}&\mapsto x_{i,r}\notag\\
b_{i,r}&\mapsto y_{i,r}.\notag
\end{align}
Moreover, by Lemma \ref{lem:rel-rule}, we have $I\subset \ker\Phi$, where $I$ is the two-sided ideal generated by $a_{i,r}b_{i+1,r}-b_{i,r}a_{i-1,r+1}$ for all $i\in[1,n]$ and $r\in \mathbb{Z}/m\mathbb{Z}$.

To complete the proof of Proposition \ref{prop:derived-class}, we compare the dimension of $\End_{\mathcal{T}}(\mu_{X_{\ell}}(A))$ with that of $A$.
Let $f=z_{1}z_{2}\cdots z_{k}$ be a morphism in $\mathcal{T}$ with $z_{1},z_{2},\ldots, z_{k}\in \mathsf{Z}$.
For the sake of simplicity, we often write down a morphism without indices, e.g., $x_{i,r}x_{i+1, r}y_{i+2,r}y_{i+1,r+1}=:xxyx=:x^{2}yx$. 
Then we can regard the morphism $f$ as a word $\mathbf{f}$ with ``$x$'' and ``$y$''. 
We denote by $x(f)$ (respectively, $y(f)$) the number of ``$x$'' (respectively, ``$y$'') in the word $\mathbf{f}$.
Note that $x(f)=y(f)=0$ if and only if $f=\mathrm{id}$.

\begin{lemma}\label{lem:dim-end-mut}
Keep the notation above.
Fix $(i,r)\in\mathbb{T}_{0}$ and the codomain of $f$ is $T(i,r)$.
Then the following statements hold.
\begin{itemize}
\item[(1)] $f\neq 0$ if and only if $(x(f),y(f))\in [0,n-i]\times [0,i-1]$. 
\item[(2)] $\mathbb{B}:=\{ x^{p}y^{q}\mid (p,q)\in [0,n-i]\times [0,i-1]\}$ forms a basis of $\Hom_{\mathcal{T}}(\mu_{X_{\ell}}(A),T(i,r))$.
\item[(3)] $\dim_{K}\Hom_{\mathcal{T}}(\mu_{X_{\ell}}(A),T(i,r))=\dim_{K}P(i,r)$.
\end{itemize}
\end{lemma}
\begin{proof}
Fix a vertex $(i,r)\in\mathbb{T}_{0}$.

(1) Let $f=z_{1}z_{2}\cdots z_{k}$ be a morphism in $\mathcal{T}$ with $z_{1},z_{2},\ldots, z_{k}\in \mathsf{Z}$ and codomain $T(i,r)$.
First we claim that if $x(f)>n-i$, then $f=0$.
For $y(f)=0$, this is clear.
In the following, we assume $y(f)\in [1,i-1]$.
By repeated use of Lemma \ref{lem:rel-rule}, we may assume the first $n-i+2$ terms of $f$ as $f=x^{n-i}yx\cdots$.
Since it follows from Lemma \ref{lem:con-rule}(4) that 
\begin{align}
x^{n-i}yx=x_{i,r}\cdots x_{n-1,r}y_{n,r}x_{n-1,r+1}=0, \notag
\end{align} 
we have $f=0$.
Similarly, we obtain that if $y(f)>i-1$, then $f=0$.
Hence the ``if'' part follows.

Next we prove the ``only if'' part.
Let $p:=x(f)\in [0,n-i]$ and $q:=y(f)\in [0,i-1]$.
By repeatedly using Lemma \ref{lem:rel-rule}, we can write $f=x^{p}y^{q}$.
It is enough to show $x^{n-i}y^{i-1}\neq 0$.
Indeed, we assume that it is true.
If $x^{p}y^{q}=0$ holds for some $p\in [0,n-i]$ and $q\in [0,i-1]$, then we have $x^{n-i}y^{i-1}=0$, a contradiction. 
We show $x^{n-i}y^{i-1}\neq 0$.
By the symmetry of the quiver, we may assume $i\in [1,\frac{n+1}{2}]$.
If $i\neq \ell$ (or equivalently $(i,r)\notin\mathcal{O}_{\ell}$ and $(n-i+1,r+i-1)\notin\mathcal{O}_{\ell}$), then we have $x^{n-i}y^{i-1}=(0,a^{n-i}b^{i-1})$.
By Lemma \ref{lem:comb-property}(1), we obtain $a^{n-i}b^{i-1}\neq 0$, and hence $x^{n-i}y^{i-1}\neq 0$.
On the other hand, if $i=\ell$ (or equivalently $(i,r)\in\mathcal{O}_{\ell}$ and $(n-i+1,r+i-1)\in\mathcal{O}_{\ell}$), then we have a commutative diagram
\begin{align}
\xymatrix{
P(n-i+1,r+i)\ar[r]^-{\left[\begin{smallmatrix} a \\ -b\end{smallmatrix}\right]}\ar[d]^{0}&{\begin{matrix}P(n-i,r+i)\\\oplus\\ P(n-i+2,r+i-1)\end{matrix}}\ar[d]^-{\left[\begin{smallmatrix} 0 & 0\\ \alpha& 0\end{smallmatrix}\right]}\\
P(i,r+1)\ar[r]^-{\left[\begin{smallmatrix} a \\ -b\end{smallmatrix}\right]}&{\begin{matrix}P(i-1,r+1)\\\oplus\\ P(i+1,r)\end{matrix}}\\
}\notag
\end{align}
where $\alpha:= a_{i+1,r}\cdots a_{n-1,r}b_{n,r}\cdots b_{n-i+1,r+i-1}$.
Suppose to the contrary that $x^{n-i}y^{i-1}=0$, that is, there exists a morphism $\left[\begin{smallmatrix}h_{1}&h_{2}\end{smallmatrix}\right]: P(n-i,r+i)\oplus P(n-i+2,r+i-1)\to P(i,r+1)$ such that
\begin{align}
h_{1}a_{n-i,r+i}-h_{2}b_{n-i+2,r+i-1}=0\label{seq:rel1}\\
\left[\begin{smallmatrix}a_{i-1,r+1}\\-b_{i+1,r}\end{smallmatrix}\right]\left[\begin{smallmatrix}h_{1}&h_{2}\end{smallmatrix}\right]=\left[\begin{smallmatrix}0&0\\\alpha&0\end{smallmatrix}\right]\label{seq:rel2}.
\end{align}
By Lemma \ref{lem:comb-property}, we can write
\begin{align}
&h_{1}=k_{1}a_{i,r+1}\cdots a_{n-2,r+1}b_{n-1,r+1}\cdots b_{n-i+2,r+i-2}b_{n-i+1,r+i-1}\notag\\
&h_{2}=k_{2}a_{i,r+1}\cdots a_{n-2,r+1}b_{n-1,r+1}\cdots b_{n-i+2,r+i-2}a_{n-i+1,r+i-1}\notag
\end{align}
with $k_{1},k_{2}\in K$.
By \eqref{seq:rel1}, we have 
\begin{align}
(k_{1}-k_{2})a_{i,r+1}\cdots a_{n-1,r+1}b_{n,r+1}\cdots b_{n-i+2,r+i-1}=0.\notag
\end{align}
Thus $k_{1}=k_{2}$ holds.
On the other hand, by \eqref{seq:rel2}, we obtain 
\begin{align}
k_{2}a_{i-1,r+1}\cdots a_{n-1,r+1}b_{n,r+1}\cdots b_{n-i+3,r+i-2}=a_{i-1,r+1}h_{2}=0.\notag
\end{align}
Hence $k_{1}=k_{2}=0$.
This implies $0=-b_{i+1,r}h_{1}=\alpha\neq 0$, a contradiction.

(2) We show that $\mathbb{B}$ gives a basis of $K$-vector space 
\begin{align}
\displaystyle \Hom_{\mathcal{T}}(\mu_{X_{\ell}}(A),T(i,r))=\bigoplus_{(j,s)\in\mathbb{T}_{0}}\Hom_{\mathcal{T}}(T(j,s), T(i,r)).\notag
\end{align}
By Lemma \ref{lem:rel-rule} and (1), $\mathbb{B}$ generates $\Hom_{\mathcal{T}}(\mu_{X_{\ell}}(A),T(i,r))$.
In the following, we show that $\mathbb{B}$ is a linear independent.
Suppose to contrary that $\mathbb{B}$ is not linear independent.
Then there are 
$\emptyset \ne V\subset [0,n-i]\times[0,i-1]$ and $\alpha_{v}\in K\setminus\{0\}$ ($\forall v\in V$) such that 
\begin{align}
\sum_{(p,q)\in V} \alpha_{(p,q)}\cdot x^{p}y^{q}=0. \notag
\end{align}
We may assume that $x^{p}y^{q}:T(j,s)\to T(i,r)$ for each $(p,q)\in V$.
Then we can choose $(p_{0}, q_{0})\in V$ such that $p_{0}<p$ and $q_{0}<q$ hold for each $(p,q)\in V\setminus\{(p_{0},q_{0})\}$.
Hence, by Lemma \ref{lem:rel-rule}, we can write  
\begin{align}
\sum_{(p,q)\in V} \alpha_{(p,q)}\cdot x^{p}y^{q}=\alpha_{(p_{0},q_{0})}x^{p_{0}}y^{q_{0}} \left(1+z\right)\ \left(z\in \rad_{\mathcal{T}}\left(T(i,r),T(i,r)\right)\right). \notag
\end{align}
This means $x^{p_{0}}y^{q_{0}}=0$.
In particular, we have
\begin{align}
x^{n-i}y^{i-1}=x^{n-i-p_{0}}y^{i-1-q_{0}}x^{p_{0}}y^{q_{0}}=0. \notag
\end{align}

(3) This follows from (2) and Lemma \ref{lem:comb-property}(2).
\end{proof}

Now we are ready to prove Proposition \ref{prop:derived-class}.
\begin{proof}[Proof of Proposition \ref{prop:derived-class}]
By Lemmas \ref{lem:local-str}(4) and \ref{lem:dim-end-mut}, there exists a surjective map $\Phi:K\mathbb{T}_{n,m}\to \End_{\mathcal{T}}(\mu_{X_{\ell}}(A))$, which induces a surjective map $A\to \End_{\mathcal{T}}(\mu_{X_{\ell}}(A))$.
Moreover, it follows from Lemma \ref{lem:dim-end-mut} that $\dim_{K}A=\dim_{K}\End_{\mathcal{T}}(\mu_{X_{\ell}}(A))$.
Hence we obtain $A\cong \End_{\mathcal{T}}(\mu_{X_{\ell}}(A))$.
\end{proof}

\subsection*{Acknowledgements}
The authors are deeply grateful to Osamu Iyama for informing them of Abe--Hoshino's work and suggesting an abstract construction of $\widetilde{A}$.

\end{document}